\documentclass[12pt,a4paper]{article}

\usepackage[english]{babel}
\usepackage[T1]{fontenc}
\usepackage[utf8]{inputenc}

\usepackage{bm}
\usepackage{lmodern}
\usepackage[normalem]{ulem}

\usepackage[top=2.25cm, bottom=2.75cm, left=2.75cm, right=2.75cm]{geometry}
\usepackage{latexsym}
\usepackage{pdflscape}
\usepackage{fancyhdr}
\usepackage{setspace}

\usepackage{graphicx}
\usepackage[dvipsnames]{xcolor}
\usepackage{epstopdf}
\usepackage{epsf}
\usepackage{eepic}
\usepackage{tikz}
\usetikzlibrary{shadings,shapes,arrows,calc,positioning,shadings,decorations.markings, patterns}

\usepackage[numbers]{natbib}

\usepackage{amsmath,amssymb,amsfonts,amscd}

\usepackage{enumitem}
\usepackage{caption}
\usepackage[format=hang,margin=10pt]{subcaption}
\usepackage{listings}
\usepackage{varwidth}
\usepackage{mathrsfs}

\usepackage{algorithm}
\usepackage{algpseudocode}
\usepackage[numbered,framed]{matlab-prettifier}
\usepackage{amsthm}
\usepackage[linguistics]{forest}
\forestset{
nice empty nodes/.style={
    for tree={l sep*=2},
    delay={where content={}{shape=coordinate}{}}
},
}

\usepackage{changepage}
\usepackage{xifthen}
\usepackage{todonotes}

\usepackage{hyperref}
\usepackage{cleveref}

\numberwithin{equation}{section}
\setlist[itemize,1]{label=$\bullet$}
\setlist[itemize,2]{label=$\triangleleft$}
\setlist[enumerate,1]{label=(\roman*)}
\setlist[enumerate,2]{label=(\arabic*)}

\definecolor{TUIl-orange}{RGB}{255, 121, 0}
\definecolor{TUIl-titleblue}{RGB}{0, 68, 121}
\definecolor{TUIl-textblue}{RGB}{0, 51, 88}
\definecolor{TUIl-green}{RGB}{0, 116, 122}
\definecolor{TUIl-grey}{RGB}{165, 165, 165}

\hypersetup{
    breaklinks=true,
    unicode=true,
    pdfpagelayout=OneColumn,
    bookmarksnumbered=true,
    bookmarksopen=true,
    bookmarksopenlevel=0,
    pdfborder={0 0 0},  
    colorlinks=true,
    linkcolor=Cerulean,    
    citecolor=Cerulean,
}

\algnewcommand\algorithmicinput{\textbf{Input:}}
\algnewcommand\AlgInput{\item[\algorithmicinput]}
\algnewcommand\algorithmicoutput{\textbf{Output:}}
\algnewcommand\AlgOutput{\item[\algorithmicoutput]}

\lstMakeShortInline"
\newtheoremstyle{dotless}{}{}{\itshape}{}{\bfseries}{}{ }{}
\newtheoremstyle{no-italic}{}{}{}{}{\bfseries}{}{ }{}
\theoremstyle{dotless}
\newtheorem{Theorem}{Theorem}[section]
\newtheorem{Example}[Theorem]{Example}
\newtheorem{Lemma}[Theorem]{Lemma}
\newtheorem{Definition}[Theorem]{Definition}
\newtheorem{Assumption}{Assumption}
\newtheorem{Remark}[Theorem]{Remark}
\newtheorem{Proposition}[Theorem]{Proposition}

\newtheorem{Test Instance}[Theorem]{Test Instance}
\theoremstyle{no-italic}

\newcommand*{\R}{\mathbb{R}}

\newcommand*{\N}{\mathbb{N}}

\DeclareMathOperator{\cl}{cl}			
\DeclareMathOperator{\argmin}{argmin}	

\usepackage{tikz}

\def\R{{\mathbb R}}

\def\WMin{\textup{WMin}}

\def\Min{\textup{Min}}

\def\Int{\textup{int }}
\def\bd{\textup{bd }}
\def\cl{\textup{cl }}
\def\conv{\textup{conv }}

\title{A Steepest Descent Method for Set Optimization Problems with Set-Valued Mappings of Finite Cardinality }
\author{Gemayqzel Bouza  \thanks{Faculty of Mathematics and Computer Science, University of Havana, 10400  Havana, Cuba,
{\texttt{gema@matcom.uh.cu}}} \and Ernest Quintana \thanks{Institute for Mathematics, Technische Universität Ilmenau, 98693 Ilmenau, Germany,
{\texttt{ernest.quintana-aparicio@tu-ilmenau.de}}} \and Christiane Tammer \thanks{Institute of Mathematics, Martin-Luther-Universität Halle-Wittenberg, 06126 Halle, Germany, {\texttt{christiane.tammer@mathematik.uni-halle.de}}}}
\date{}

\begin{document}

\maketitle

\begin{abstract}
In this paper, we study a first order solution method for a particular class of set optimization problems where the solution concept is given by the set approach. We consider the case in which the set-valued objective mapping is identified by a finite number of continuously differentiable selections. The corresponding set optimization problem is then equivalent to find optimistic solutions to vector optimization problems under uncertainty with a finite  uncertainty set. We develop optimality conditions for these types of problems, and introduce two concepts of critical points. Furthermore, we propose a descent method and provide a convergence result to points satisfying the optimality conditions previously derived. Some numerical examples illustrating the performance of the method are also discussed.  This paper is a modified and polished version of Chapter 5 in the PhD thesis by Quintana (On set optimization with set relations: a scalarization approach to optimality conditions and algorithms, Martin-Luther-Universität Halle-Wittenberg, 2020).
\end{abstract}

\noindent {\small\textbf{Key Words:} set optimization, robust vector optimization, descent method, stationary point}

\vspace{2ex} \noindent {\small\textbf{Mathematics subject
classifications (MSC 2010):}}  	90C29, 90C46, 90C47	


\section{Introduction}

Set optimization is the class of mathematical problems that consists in minimizing set-valued mappings acting between two vector spaces, in which the image space is partially ordered by a given closed, convex and pointed cone. There are two main approaches for defining solution concepts for this type of problems, namely the vector approach and the set approach. In this paper, we deal with the last of these concepts. The main idea of this approach lies on defining a preorder on the power set of the image space, and to consider  minimal solutions of the set-valued problem accordingly. Research in this area started with the works of Young \cite{Young1931}, Nishnianidze \cite{Nishnianidze1984}, and  Kuroiwa \cite{kuroiwa1998, Kuroiwa2001}, in which the first set relations for defining a preorder were considered. Furthermore, Kuroiwa \cite{Kuroiwa1997firstsolutiondef} was the first who considered set optimization problems where the solution concept is given by the set approach. Since then, research in this direction has expanded immensely due to its applications in finance, optimization under uncertainty, game theory, and socioeconomics.  We refer the reader to \cite{KTZ} for a comprehensive overview of the field.

The research topic that concerns us in this paper is the development of efficient algorithms for the solution of set optimization problems. In this setting, the current approaches in the literature can be roughly clustered into four different groups:

\begin{itemize}
\item Derivative free methods \cite{jahn2015desc, jahn2018tree, KK2016}.

In this context, the derived  algorithms are descent methods and use a derivative free strategy \cite{ConnScheinbergVicente2009}. These algorithms are designed to deal with unconstrained problems, and they assume no particular structure of the set-valued objective mapping. The first method of this type was described in \cite{jahn2015desc}. There, the case in which both the epigraphical and hypographical multifunctions of the set-valued objective mapping have convex values was analyzed. This convexity assumption was then relaxed in \cite{KK2016} for the so called upper set less relation. Finally, in \cite{jahn2018tree}, a new method with this strategy was studied. An interesting feature of the algorithm in this reference is that, instead of choosing only one descent direction at every iteration, it considers several of them  at the same time. Thus, the method generates a tree with the initial point as the root, and the possible solutions as leaves.
\item  Algorithms of a sorting type \cite{guntherkobispopovici2019, guntherkobispopovici2019part2,  kobiskuroiwatammer2017, kobistam2018}.

The methods in this class are specifically  designed to treat set optimization problems with a finite feasible set. Because of this, they are based on simple comparisons between the images of the set-valued objective mapping. In \cite{kobiskuroiwatammer2017, kobistam2018}, the algorithms are extensions of those by Jahn  \cite{jahn2006subdiv,jahnrathje2006} for vector optimization problems. They use a so called forward and backward reduction procedures that, in practice, avoid making many of these previously mentioned comparisons. Therefore, these methods perform more efficiently than a naive implementation in which every pair of sets must be compared. More recently, in \cite{guntherkobispopovici2019, guntherkobispopovici2019part2}, an extension of the algorithm by G\"{u}nther and Popovici \cite{guntherpopovici2018} for vector problems was studied. The idea now is to, first, find  an enumeration of the images of the set-valued mapping whose values by a scalarization using a strongly monotone functional are increasing. In a second step, a forward iteration procedure is performed. Due to the presorting step, these methods enjoy an almost optimal computational complexity, compare \cite{KungLuccioPreparata1975}.

\item Algorithms based on scalarization \cite{EhrgottIdeSchobel2014,eichfeldernieblingrocktaschel2019,  idekobis2014, IdeKobisKuroiwa2014, jiangcaoxiong2019, schmidtschobelthom2019}.

The methods in this group follow a scalarization approach, and are derived for problems where the set-valued objective mapping has a particular structure that comes from the so called robust counterpart of a vector optimization problem under uncertainty, see \cite{IdeKobisKuroiwa2014}. In \cite{EhrgottIdeSchobel2014, idekobis2014, IdeKobisKuroiwa2014}, a linear scalarization was employed for solving the set optimization problem. Furthermore, the  $\epsilon$- constraint method was extended too in \cite{EhrgottIdeSchobel2014, idekobis2014}, for the particular case in which the ordering cone is the nonnegative orthant. Weighted Chebyshev scalarization and some of its variants (augmented, min-ordering) were also studied in \cite{idekobis2014, jiangcaoxiong2019, schmidtschobelthom2019}. 

\item Branch and bound \cite{eichfeldernieblingrocktaschel2019}.

The algorithm in \cite{eichfeldernieblingrocktaschel2019} is also designed for uncertain vector optimization problems, but in particular it is assumed that only the decision variable is the source of uncertainty. There, the authors propose a branch and bound method for finding a box covering of the solution set. 
\end{itemize} 

The strategy that we consider in this paper is different to the ones previously described, and is designed for dealing with unconstrained set optimization problems in which the set-valued objective mapping is given by a finite number of continuously differentiable selections. Our motivation for studying problems with this particular structure is twofold:

\begin{itemize}
\item Problems of this type have important applications in optimization under uncertainty.

Indeed, set optimization problems with this structure arise when computing robust solutions to vector optimization problems under uncertainty, if the so called uncertainty set is finite, see  \cite{IdeKobisKuroiwa2014}. Furthermore, the solvability  of problems with a finite uncertainty set is an important component in the treatment of the general case with an infinite uncertainty set, see the cutting plane strategy in \cite{mutapcicboyd2009} and the reduction results in  \cite[Proposition 2.1]{bentalnemirovski1998} and \cite[Theorem 5.9]{EhrgottIdeSchobel2014}.

\item Current algorithms in the literature pose different theoretical and practical difficulties when solving these types of problems.

Indeed, although derivative free methods can be directly applied in this setting, they suffer from the same drawbacks as their counterparts in the scalar case. Specifically, because they make no use of first order information (which we assume is available in our context), we expect them to perform slower in practice that a method who uses these additional properties. Even worse, in the set-valued setting, there is now an increased cost on performing comparisons between sets, which was almost negligible for scalar problems. On the other hand, the algorithms of a sorting type described earlier can not be used in our setting since they require a finite feasible set. Similarly, the branch and bound strategy is designed for problems that do not fit the particular structure that we consider in this paper, and so it can not be taken into account. Finally, we can also consider the algorithms based on scalarization in our context. However, the main drawback of these methods is that, in general, they are not able to recover all the solutions of the set optimization problem. In fact, the $\epsilon$- constraint method, which is known to overcome this difficulty in standard multiobjective optimization, will fail in this setting. 
\end{itemize} Thus, we address in this paper the need of a first order method that exploits the particular structure of the set-valued objective mapping previously mentioned, and does not have the same drawbacks of the other approaches in the literature. 

The rest of the paper is structured as follows. We start in Section \ref{sec: preliminaries} by introducing the main notations, basic concepts and results that will be used throughout the paper. In Section \ref{sec:oc algo}, we derive optimality conditions for set optimization problems with the aforementioned structure. These optimality conditions  constitute the basis of the descent method described in Section \ref{sec:algo algo}, where the full convergence of the algorithm is also obtained.  In Section \ref{sec:numerical algo}, we illustrate the performance of the method on different test instances. We conclude in Section 6 by summarizing our results and proposing ideas for further research.

\section{Preliminaries}\label{sec: preliminaries}
We start this section by introducing the main notations used in the paper. First, the class of all nonempty subsets of $\R^m$ will be denoted by $\mathscr{P}(\R^m).$ Furthermore, for $A \in \mathscr{P}(\R^m)$, we denote by $\Int A$, $\cl A$, $\bd A$ and $\conv A$  the interior, closure, boundary and convex hull of the set $A,$ respectively. All the considered vectors  are column vectors, and we denote the transpose operator with the symbol $\top.$ On the other hand,  $\|\cdot\|$ will stand for either the euclidean norm of a vector or for the standard spectral norm of a matrix, depending on the context.  We also denote the cardinality of a finite set $A$  by $|A|.$ Finally, for $k\in \N,$ we put $[k] = \{1,\ldots,k\}.$

We next consider the most important definitions and properties  involved in the results of the paper. Recall that a set $K \in \mathscr{P}(\R^m)$ is said to be a cone if $t y\in K$ for every $y\in K$ and every $t \geq 0.$ Moreover, a cone $K$ is called convex if $K + K = K,$ pointed if $K\cap (-K)=\{0\},$ and solid if $\Int K \neq \emptyset.$ An important related concept is that of the dual cone. For a cone $K,$ this is the set 

\begin{equation*}
K^*:=\{v \in \R^m \mid \forall\; y\in K: v^\top y\geq 0\}.
\end{equation*} Throughout, we suppose that $K \in \mathscr{P}(\R^m)$ is a cone.

It is well known \cite{GRTZ} that when $K$ is convex and pointed, it generates a partial order $\preceq$ on $\R^m$ as follows: 

\begin{equation}\label{eq:porder}
y\preceq z: \Longleftrightarrow z-y\in K.
\end{equation}
Furthermore, if $K$ is solid, one can also consider the so called strict order $\prec$ which is defined by

\begin{equation}\label{eq:sporder}
y\prec z: \Longleftrightarrow z-y\in \Int K.
\end{equation} In the following definition, we collect the concepts of minimal and weakly minimal elements of a set with respect to $\preceq.$
\begin{Definition}\label{def:wminimalconceptsvop}

Let  $A \in \mathscr{P}(\R^m)$ and suppose that $K$ is closed, convex, pointed, and solid.  

\begin{enumerate}

\item  The set of minimal elements of $A$ with respect to $K$ is defined as

$$\Min (A,K):= \{y \in A\mid \left(y-  K\right)\cap A =\{y\}\}.$$

\item The set of weakly minimal elements of $A$ with respect to $K$ is defined as

$$\WMin (A,K):= \{y \in A\mid \left(y- \Int K\right)\cap A =\emptyset\}.$$
\end{enumerate} 

\end{Definition}

The following proposition will be often used.

\begin{Proposition}(\cite[Theorem $6.3 \;c)$]{Jahn2011})\label{prop:domination property}
Let $A \in \mathscr{P}(\R^m)$ be nonempty and compact, and  $K$ be closed, convex, and pointed. Then, $A$ satisfies the so called domination property with respect to $K$, that is, $\Min (A,K)\neq \emptyset$ and $$A + K = \Min(A,K) + K.$$
\end{Proposition}

The Gerstewitz scalarizing functional will play also an important role in the main results. 

\begin{Definition}\label{def:funcGW}
Let $K$ be closed, convex, pointed, and solid. For a given element $e \in \Int K,$ the Gerstewitz functional associated to $e$ and $K$ is $\psi_e: \R^m \rightarrow \R$ defined as

\begin{equation}
\psi_e(y):= \min\{t\in \mathbb{R} \mid te\in y+K\}.
\end{equation}
\end{Definition}

Useful properties of this functional are summarized in the next proposition.

\begin{Proposition}(\cite[Section 5.2]{KTZ})\label{prop: properties of tammer function} Let $K$ be  closed, convex, pointed and solid, and consider an element $e \in \Int K.$ Then, the functional $\psi_e$ satisfies the following properties: 
 
\begin{enumerate}

\item \label{item:sublinearity and lipschitz of scalarization} $\psi_e$ is sublinear and Lipschitz on $\R^m.$ 

\item \label{item:monotonicity of scalarization} $\psi_e$ is both monotone and strictly monotone with respect to the partial order $\preceq$, that is, 

$$\forall \; y,z \in \R^m:  y \preceq z \Longrightarrow \psi_e(y)\leq \psi_e(z)$$ and 

$$\forall \; y,z \in \R^m:  y \prec z \Longrightarrow \psi_e(y) < \psi_e(z),$$ respectively.

\item \label{item:representability of scalarization} $\psi_e$ satisfies the so called representability property, that is,
$$- K= \{y \in \R^m \mid \psi_e(y)\leq 0 \}, \quad - \Int K = \{y \in \R^m \mid \psi_e(y)< 0 \}.$$
\end{enumerate}

\end{Proposition}

We next introduce the set relations between the nonempty subsets of $\R^m$ that will be used in the definition of the set optimization problem we consider. We refer the reader to \cite{jahnha2011, karamanemrahetal2018} and the references therein for other set relations. 

\begin{Definition}(\cite{KTH})\label{def-set-relation}
For the given cone $K,$ the lower set less relation $\preceq^\ell$ is the binary relation defined on $\mathscr{P}(\R^m)$ as follows: 

$$ \forall\; A,B \in \mathscr{P}(\R^m):  A\preceq^\ell B: \Longleftrightarrow B\subseteq A+K.$$ Similarly, if $K$ is solid, the strict lower set less relation $\prec^\ell$ is the binary relation defined on $\mathscr{P}(\R^m)$ by:

$$ \forall\; A,B \in \mathscr{P}(\R^m):  A\prec^\ell B: \Longleftrightarrow B\subseteq A+ \Int K.$$

\end{Definition}
\begin{Remark}\label{rem:set relation=partial order}
Note that for any two vectors $y,z \in \R^m$ the following equivalences hold:

$$ \{y\} \preceq^\ell \{z\} \Longleftrightarrow y \preceq z,   \;\; \{y\} \prec^\ell \{z\} \Longleftrightarrow y \prec z.$$ Thus, the restrictions of $\preceq^\ell$ and $\prec^\ell$ to the singletons in $\mathscr{P}(\R^m)$ are equivalent  to $\preceq$ and $\prec,$ respectively.

\end{Remark}

We are now ready to present the set optimization problem together with a solution concept based on set relations. 

\begin{Definition}
Let $F:\R^n \rightrightarrows \R^m$ be a given set-valued mapping and suppose that $K$ is closed, convex, pointed, and solid. The set optimization problem with this data is formally represented as 

\begin{equation}\label{eq: SP}
 \begin{array}{ll}
\preceq^\ell\textrm{- }\min\limits_{ x\in \R^n} \; F(x), \tag{$\mathcal{SP}_\ell$}\\
\end{array} 
\end{equation} and a solution is understood in the following sense: we say that a point $\bar{x} \in \R^n$ is a local weakly minimal solution  of \eqref{eq: SP} if there exists a neighborhood $U$ of $\bar{x}$ such that the following holds:

$$\nexists \; x \in U: F(x) \prec^\ell F(\bar{x}).$$ Moreover, if we can choose $U = \R^n$ above, we simply say that $\bar{x}$ is a weakly minimal solution  of \eqref{eq: SP}. 
\end{Definition}

\begin{Remark}
A related problem to \eqref{eq: SP} that is relevant in our paper is the so called vector optimization problem \cite{Jahn2011,Luc2}. There, for a vector-valued mapping $f: \R^n \rightarrow \R^m,$ one considers 

\begin{equation*}
 \begin{array}{ll}
\preceq\textrm{- }\min\limits_{ x\in \R^n} \; f(x), \\
\end{array} 
\end{equation*} where a point $\bar{x}$ is said to be a weakly minimal solution if $f(\bar{x}) \in \WMin(f[\R^n],K)$ (corresponding to Definition \ref{def:wminimalconceptsvop}). Taking into account Remark \ref{rem:set relation=partial order}, it is easy to verify that this solution concept coincides with ours for \eqref{eq: SP} when the set-valued mapping $F$ is given by $F(x):= \{f(x)\}$ for every $x \in \R^n.$  
\end{Remark}

We conclude the section by establishing the main assumption employed in the rest of the paper for the treatment of \eqref{eq: SP}:

\begin{Assumption}\label{ass algorithm} 
Suppose that $K \in \mathscr{P}(\R^m)$ is a closed, convex, pointed and solid cone, and that $ e\in \Int K$ is fixed. Furthermore, consider a reference point  $\bar{x} \in \R^n,$ given vector-valued functions $f^1, f^2,\ldots, f^p: \R^n \rightarrow \R^m$ that are continuously differentiable, and assume that the set-valued mapping $F$ in \eqref{eq: SP} is defined by 

\begin{equation*}
F(x)=\bigg\{f^1(x), f^2(x),\ldots, f^p(x) \bigg\}.
\end{equation*} 
\end{Assumption}  

\section{Optimality Conditions}\label{sec:oc algo}

In this section, we study optimality conditions for weakly minimal solutions of \eqref{eq: SP} under Assumption \ref{ass algorithm}. These conditions are the foundation on which the proposed algorithm is built. In particular, because of the resemblance of our method with standard gradient descent in the scalar case, we are interested in Fermat rules for set optimization problems. Recently, results of this type were derived in \cite{bouzaquintanatuantammer2020}, see also \cite{amahroqoussarhan2019}. There, the optimality conditions involve the computation of the limiting normal cone \cite{Mordukhovich1} of the set-valued mapping $F$ at different points in its graph. However, this is a difficult task in our case because the graph of $F$  is the union of the graphs of the vector-valued functions $f^i,$ and to the best of our knowledge there is no exact formula for finding the normal cone to the union of sets (at a given point) in terms of the initial data. Thus, instead of considering the results from  \cite{bouzaquintanatuantammer2020}, we  exploit the particular structure of $F$ and the differentiability of the functionals $f^i$ to deduce new necessary conditions.

We start by defining some index-related set-valued mappings that will be of importance. They make use of the concepts introduced  in Definition \ref{def:wminimalconceptsvop}.

\begin{Definition}\label{def:index set valued mappings}
The following set-valued mappings are defined:

\begin{enumerate}
\item The active index of minimal elements associated to $F$ is $I:\R^n \rightrightarrows [p]$ given by $$I(x):= \big\{i \in  [p] \mid f^i(x) \in \Min (F(x),K) \big\}.$$

\item The active index of weakly minimal elements associated to $F$ is $I_0:\R^n \rightrightarrows [p]$ defined as  $$I_0(x):= \big\{i \in  [p] \mid f^i(x) \in \WMin (F(x),K) \big\}.$$

\item For a vector $v\in \R^m,$ we define $I_v:\R^n \rightrightarrows [p]$ as 

$$I_v(x):= \{i \in I(x) \mid f^i(x) = v\}.$$ 
\end{enumerate}
\end{Definition} 
 
It follows directly from the definition that  $I_v(x) = \emptyset$ whenever $v \notin \Min (F(x),K)$ and that 

\begin{equation}\label{eq:partition of I}
\forall \; x \in \R^n: I(x) = \bigcup\limits_{v \in \Min(F(x),K)} I_v(x).
\end{equation}

\begin{Definition}
The map $\omega:\R^n \rightarrow \R$ is defined as the cardinality of the set of minimal elements of $F$, that is, 

$$\omega(x): = |\Min (F(x),K)|.$$ Furthermore, we set $\bar{\omega}:= \omega(\bar{x}).$
\end{Definition}

From now on we consider that, for any point $x \in \R^n,$ an enumeration $\{v^x_1,\ldots, v^x_{\omega(x)}\}$ of the set $\Min (F(x),K)$ has been chosen in advance.

\begin{Definition}
For a given point $x\in \R^n,$ consider the enumeration $\{v^x_1,\ldots, v^x_{\omega(x)}\}$ of the set $\Min (F(x),K).$ The partition set of $x$ is defined as 

$$P_x:= \prod\limits_{j=1}^{\omega(x)}I_{v^x_j}(x),$$ where $I_{v^x_j}(x)$ is given in Definition \ref{def:index set valued mappings} $(iii)$ for $j \in [\omega(x)].$

\end{Definition}

The optimality conditions for \eqref{eq: SP} we will present are based on the following idea: from the particular structure of $F,$ we will construct a family of vector optimization problems that, together, locally represent \eqref{eq: SP} (in a sense to be specified) around the point which must be checked for optimality. Then, (standard) optimality conditions are applied to the family of vector optimization problems. The following lemma is the key step in that direction. 

\begin{Lemma}\label{lem: equivalent vector problem lower}

Let  $\tilde{K} \in \mathscr{P}\left(\R^{m \bar{\omega}} \right)$ be the cone defined as

\begin{equation}\label{eq:tildeK}
\tilde{K}:= \prod\limits_{j=1}^{\bar{\omega}} K,
\end{equation} and let us denote by $\preceq_{\tilde{K}}$ and $\prec_{\tilde{K}}$  the partial order and the strict order in $\R^{m\bar{\omega}}$ induced by $\tilde{K},$ respectively (see \eqref{eq:sporder}). Furthermore, consider the partition set $P_{\bar{x}}$ associated to $\bar{x}$ and define, for every  $a \in P_{\bar{x}},$ the functional $\tilde{f}^a: \R^n \rightarrow \prod\limits_{j=1}^{\bar{\omega}} \R^m$ as 

\begin{equation}
\tilde{f}^a(x):= \begin{pmatrix}
 f^{a_1}(x)\\ \vdots \\ f^{a_{\bar{\omega}}}(x)
\end{pmatrix}.
\end{equation} Then, $\bar{x}$ is a local weakly minimal solution of \eqref{eq: SP} if and only if, for every $a \in P_{\bar{x}},$ $\bar{x}$ is a local weakly minimal solution of the vector optimization problem 
\begin{equation}\label{eq: associated vector opt problem lower}
 \begin{array}{ll}
\preceq_{\tilde{K}} \textrm{- }\min\limits_{x\in \R^n} \; \tilde{f}^a(x). \tag{$\mathcal{VP}_a$}\\

\end{array} 
\end{equation}

\end{Lemma}
\begin{proof}We argue by contradiction in both cases. First, assume that $\bar{x}$ is a local weakly minimal solution of \eqref{eq: SP} and that, for some $a\in P_{\bar{x}},$  $\bar{x}$ is not a local weakly minimal solution of \eqref{eq: associated vector opt problem lower}. Then, we could find a sequence $\{x_k\}_{k\geq 1}\subset \R^n$  such that $x_k \to \bar{x}$ and 

\begin{equation}\label{eq:vpa minimizing sequence}
\forall \; k \in \N: \tilde{f}^a(x_k) \prec_{\tilde{K}} \tilde{f}^a(\bar{x}).
\end{equation} Hence, we deduce that

\begin{eqnarray*}
\forall \; k \in \N: F(\bar{x}) & \overset{(\textrm{Proposition } \ref{prop:domination property})}{\subseteq} & \{f^{a_1}(\bar{x}),\ldots, f^{a_{\bar{\omega}}} (\bar{x})\} + K\\
           & \overset{\eqref{eq:vpa minimizing sequence}}{\subseteq }&  \{f^{a_1}(x_k),\ldots, f^{a_{\bar{\omega}}} (x_k)\} + \Int K+ K\\
           & \subseteq & F(x_k) + \Int K.
\end{eqnarray*} Since this is equivalent to $F(x_k) \prec^\ell F(\bar{x})$ for every $k \in \N$ and $x_k \to \bar{x},$ we contradict the weak minimality of $\bar{x}$ for \eqref{eq: SP}.

Next, suppose that $\bar{x}$ is  a local weakly minimal solution of \eqref{eq: associated vector opt problem lower} for every $a \in P_{\bar{x}}$, but not a local weakly minimal solution of \eqref{eq: SP}. Then, we could find a sequence $\{x_k\}_{k\geq 1} \subset \R^n$ such that $x_k \to \bar{x}$ and $F(x_k) \prec^\ell F(\bar{x})$ for every $k\in \N.$ Consider the enumeration $\{v^{\bar{x}}_1,\ldots, v^{\bar{x}}_{\bar{\omega}}\}$ of the set $\Min (F(\bar{x}),K).$ Then, 

\begin{equation}\label{eq:intermediate ineq jik}
\forall \; j\in [\bar{\omega}], k \in \N, \exists \; i_{(j,k)} \in [p]:f^{i_{(j,k)}}(x_k)\prec v^{\bar{x}}_j.
\end{equation} Since the indexes $i_{(j,k)}$ are being chosen on the finite set $[p]$, we can assume without loss of generality that $i_{(j,k)}$ is independent of $k,$ that is, $i_{(j,k)} = \bar{i}_j$ for every $k\in \N$ and some $\bar{i}_j \in [p].$ Hence, taking the limit in \eqref{eq:intermediate ineq jik} when $k \to + \infty$, we get 

\begin{equation}\label{eq:interm 2}
\forall\; j \in  [\bar{\omega}]: f^{\bar{i}_j}(\bar{x})\preceq v^{\bar{x}}_j.
\end{equation} Because $v^{\bar{x}}_j \in \Min (F(\bar{x}),K),$ it follows from \eqref{eq:interm 2} that $f^{\bar{i}_j}(\bar{x})= v^{\bar{x}}_j$ and that $\bar{i}_j \in I(\bar{x})$ for every $j \in  [\bar{\omega}].$ Consider now the tuple $\bar{a}:= (\bar{i}_1,\ldots ,\bar{i}_{\bar{\omega}}).$ Then, it can be verified that $\bar{a}\in P_{\bar{x}}.$ Moreover, from  \eqref{eq:intermediate ineq jik} we deduce that $\tilde{f}^{\bar{a}}(x_k) \prec_{\tilde{K}} \tilde{f}^{\bar{a}}(\bar{x})$ for every $k \in \N.$ Since $x_k \to \bar{x}$, this contradicts the weak minimality of $\bar{x}$ for \eqref{eq: associated vector opt problem lower} when $a = \bar{a}.$ 
\end{proof}

We now establish the necessary optimality condition for \eqref{eq: SP} that will be used  in our descent method.
\begin{Theorem}\label{thm:oc sopt finite}

Suppose that $\bar{x}$ is a local weakly minimal solution of \eqref{eq: SP}. Then,

\begin{equation}\label{eq:oc soptfinite}
\forall \; a \in P_{\bar{x}}, \; \exists\; \mu_1, \mu_2, \ldots , \mu_{\bar{w}} \in K^* : \sum_{j=1}^{\bar{\omega
}}\nabla f^{a_j}(\bar{x})\mu_j=0,\; (\mu_1,\ldots,\mu_{\bar{w}})\neq 0. 
\end{equation} Conversely, assume that $f^i$ is $K$-  convex for each $i \in I(\bar{x})$, that is, 

$$\forall\; i \in I(\bar{x}), x_1,x_2 \in \R^n, t \in [0,1]: f^i(t x_1 +(1-t)x_2) \preceq tf^i(x_1)+ (1-t)f^i(x_2).$$

Then, condition  \eqref{eq:oc soptfinite} is also sufficient for the  local weak minimality of $\bar{x}.$ 
\end{Theorem}

\begin{proof}
By Lemma \ref{lem: equivalent vector problem lower}, we get that $\bar{x}$ is a local weakly minimal solution of \eqref{eq: associated vector opt problem lower} for every $a\in P_{\bar{x}}.$ Applying now \cite[Theorem 4.1]{Luc2} for every $a\in P_{\bar{x}}$, we get 

\begin{equation}\label{eq:oc sp2}
\forall \; a\in P_{\bar{x}},\exists \; \mu \in \tilde{K}^*\setminus\{0\}: \nabla \tilde{f}^a(\bar{x})\mu =0.
\end{equation} Since $\tilde{K}^* = \prod\limits_{j=1}^{\bar{\omega}} K^*,$ it is easy to verify that  \eqref{eq:oc sp2} is equivalent to the first part of the statement

In order to see the sufficiency under convexity, assume that $\bar{x}$ satisfies \eqref{eq:oc soptfinite}. Note that for any $a \in P_{\bar{x}}$, the function $\tilde{f}^a$ is $\tilde{K}$-  convex provided that each $f^i$ is $K$-  convex for every $i \in I(\bar{x}).$ Then, in this case, it is well known that \eqref{eq:oc sp2} is equivalent to $\bar{x}$ being a local weakly minimal solution of  \eqref{eq: associated vector opt problem lower} for every $a \in P_{\bar{x}},$ see \cite{DrummondSvaiter2005}. Applying now Lemma \ref{lem: equivalent vector problem lower}, we obtain that $\bar{x}$ is a local weakly minimal solution of \eqref{eq: SP}. 
\end{proof}
Based on Theorem \ref{thm:oc sopt finite}, we define the following concepts of stationarity for \eqref{eq: SP}.
 
\begin{Definition}\label{def:stat sop2}
We say that $\bar{x}$ is a   stationary point of \eqref{eq: SP} if there exists a nonempty set $Q \subseteq P_{\bar{x}}$ such that the following assertion holds:
\begin{equation}\label{eq: stationarity lower def}
\forall \; a \in Q, \; \exists\; \mu_1, \mu_2, \ldots , \mu_{\bar{w}} \in K^* : \sum_{j=1}^{\bar{\omega
}}\nabla f^{a_j}(\bar{x})\mu_j=0,\; (\mu_1,\ldots,\mu_{\bar{w}})\neq 0. 
\end{equation} In that case, we also say that $\bar{x}$ is   stationary with respect to $Q.$ If, in addition, we can choose $Q = P_{\bar{x}}$ in \eqref{eq: stationarity lower def}, we simply  call $\bar{x}$ a strongly   stationary point. 
\end{Definition}

\begin{Remark}\label{rem:weak stat is stat for vp} It follows from Definition \ref{def:stat sop2} that a strongly stationary point of \eqref{eq: SP} is also stationary with respect to $Q$ for every $Q \subseteq P_{\bar{x}}.$ Furthermore, from Theorem \ref{thm:oc sopt finite}, it is clear that stationarity is also a necessary optimality condition for \eqref{eq: SP}.
\end{Remark}

In the following example, we illustrate a comparison of our optimality conditions  with  previous ones in the literature for standard optimization problems. 

\begin{Example}
Suppose that in Assumption \ref{ass algorithm} we have $m = 1, K = \R_+.$   Furthermore, consider the functional $f : \R^n \rightarrow \R$ defined as $$f(x) := \min\limits_{i \in [p]} f^i(x)$$ and problem \eqref{eq: SP} associated to this data. Hence, in this case, $$P_{\bar{x}} = I(\bar{x}) = \{i \in [p] \mid f^i(\bar{x}) = f(\bar{x})\}.$$    It is then easy to verify that the following statements hold:
\begin{enumerate}

\item $\bar{x}$ is strongly   stationary for \eqref{eq: SP} if and only if 

$$\forall \; i \in I(\bar{x}) : \nabla f^i(\bar{x}) = 0.$$ 

\item $\bar{x}$ is   stationary for \eqref{eq: SP} if and only if 

$$\exists \; i \in I(\bar{x}) : \nabla f^i(\bar{x}) = 0.$$

\end{enumerate} On the other hand, it is straightforward to verify that $\bar{x}$ is a weakly minimal solution of  \eqref{eq: SP} if and only if $\bar{x}$ is a solution of the problem 

\begin{equation}\label{eq:P}
 \begin{array}{ll}
\min\limits_{ x\in \R^n} \; f(x). \tag{$\mathcal{P}$}\\
\end{array} 
\end{equation} Moreover, if we denote by $\widehat{\partial} f(\bar{x})$ and $\partial f (\bar{x})$ the Fr\'{e}chet and Mordukhovich subdifferential  of $f$ at the point $\bar{x}$  respectively (see \cite{Mordukhovich1}),  it follows from \cite[Proposition 1.114]{Mordukhovich1} that the inclusions 

\begin{equation}\label{eq:stat frechet chp5}
0 \in  \widehat{\partial} f(\bar{x})
\end{equation} and

\begin{equation}\label{eq:mordu stat chp5}
 0\in  \partial f (\bar{x})
\end{equation} are necessary for  $\bar{x}$ being a solution of \eqref{eq:P}. A point $\bar{x}$ satisfying \eqref{eq:stat frechet chp5} and  \eqref{eq:mordu stat chp5} is said to be Fr\'{e}chet and Mordukhovich stationary for \eqref{eq:P}, respectively.  Furthermore, from \cite[Proposition 5]{eberhardroschina2019} and  \cite[Proposition 1.113]{Mordukhovich1}, we have 

\begin{equation}\label{eq:frechet of minimum}
\widehat{\partial} f(\bar{x}) = \bigcap_{i \in I(\bar{x})} \{\nabla f^i(\bar{x})\}
\end{equation} and

\begin{equation}\label{eq:mordu of minimum}
\partial f(\bar{x}) \subseteq \bigcup_{i \in I(\bar{x})} \{\nabla f^i(\bar{x})\} 
\end{equation} respectively. Thus, from \eqref{eq:stat frechet chp5}, \eqref{eq:frechet of minimum} and $(i)$, we deduce that 

\begin{enumerate}
\setcounter{enumi}{2}
\item $\bar{x}$ is strongly   stationary for \eqref{eq: SP} if and only if $\bar{x}$ is Fr\'{e}chet stationary for \eqref{eq:P}. 
\end{enumerate}

Similarly, from \eqref{eq:mordu stat chp5}, \eqref{eq:mordu of minimum} and $(ii)$, we find that

\begin{enumerate}
\setcounter{enumi}{2}
\item If $\bar{x}$ is Mordukhovich stationary for \eqref{eq:P}, then $\bar{x}$ is   stationary for \eqref{eq: SP}. 
\end{enumerate}
\end{Example}

We close the section with the following proposition, that presents an alternative characterization of   stationary points.

\begin{Proposition}
Let $Q \subseteq P_{\bar{x}}$ be given. Then, $\bar{x}$ is   stationary for \eqref{eq: SP} with respect to $Q$ if and only if

\begin{equation}\label{eq:stat alternative}
\forall \; a \in Q, u \in \R^n, \exists \;j \in  [\bar{\omega}] : \nabla f^{a_j}(\bar{x})^\top u \notin - \Int K.
\end{equation}

\end{Proposition}
\begin{proof}
Suppose first that $\bar{x}$ is   stationary with respect to $Q.$ Fix now $a \in Q, u \in \R^n,$ and consider the vectors $\mu_1,\mu_2,\ldots, \mu_{\bar{\omega}} \in K^*$ that satisfy \eqref{eq: stationarity lower def}. We argue  by contradiction. Assume that

\begin{equation}\label{eq:nablaftd intk}
\forall \;j \in  [\bar{\omega}]: \nabla f^{a_j}(\bar{x})^\top u \in - \Int K.
\end{equation} From \eqref{eq:nablaftd intk}  and the fact that $\left(\mu_1,\ldots,\mu_{\bar{\omega}}\right) \in \left( \prod\limits_{j=1}^{\bar{\omega}} K^* \right) \setminus\{0\},$ we deduce that 

\begin{equation}\label{eq:aux nonzero}
\left(\mu_1^\top \left(\nabla f^{a_1}(\bar{x})^\top u\right), \ldots, \mu_j^\top \left(\nabla f^{a_j}(\bar{x})^\top u\right)  \right) \in - \R^{\bar{\omega}}_+ \setminus \{0\}.
\end{equation} Hence, we get

$$0 \overset{\eqref{eq: stationarity lower def}}{=} \left(\sum_{j=1}^{\bar{\omega}}\nabla f^{a_j}(\bar{x})\mu_j\right)^\top u =\sum_{j=1}^{\bar{\omega}} \mu_j^\top \left(\nabla f^{a_j}(\bar{x})^\top u \right) \overset{\eqref{eq:aux nonzero}}{<}0,$$ a contradiction.

Suppose now that \eqref{eq:stat alternative} holds, and fix $a \in Q.$ Consider the functional $\tilde{f}^a$ and the cone $\tilde{K}$ from Lemma \ref{lem: equivalent vector problem lower}, together with the set 

$$A:=  \left\{ \nabla \tilde{f}^a(\bar{x})^{\top} u \mid u \in \R^n \right\}.$$ Then, we deduce from \eqref{eq:stat alternative} that 

$$A \cap \Int \tilde{K} = \emptyset.$$ Applying now Eidelheit's separation theorem for convex sets \cite[Theorem 3.16]{Jahn2011}, we obtain $(\mu_1, \ldots, \mu_{\bar{\omega}}) \in \left( \prod\limits_{j=1}^{\bar{\omega}} \R^m \right) \setminus\{0\}$ such that

\begin{equation}\label{eq:ineq separation theorem}
\forall\; u \in \R^n, v_1, \ldots,v_{\bar{\omega}} \in K: \left(\sum\limits_{j=1}^{\bar{\omega}} \nabla f^{a_j}(\bar{x})\mu_j  \right)^{\top} u \leq \sum\limits_{j=1}^{\bar{\omega}} \mu_j^\top v_j.
\end{equation} By fixing $\bar{j} \in  [\bar{\omega}]$ and substituting $u = 0, v_j = 0$ for $j \neq \bar{j}$ in \eqref{eq:ineq separation theorem}, we obtain 

$$\forall\; v_{\bar{j}} \in K: \mu_{\bar{j}}^\top v_{\bar{j}}  \geq 0.$$ Hence, $\mu_{\bar{j}} \in K^*.$ Since $\bar{j}$ was chosen arbitrarily, we get that $(\mu_1, \ldots, \mu_{\bar{\omega}}) \in \left( \prod\limits_{j=1}^{\bar{\omega}} K^* \right) \setminus\{0\}.$ Define now 

$$\bar{u}:= \sum\limits_{j=1}^{\bar{\omega}} \nabla f^{a_j}(\bar{x})\mu_j .$$ Then, to finish the proof, we need to show that $\bar{u}= 0.$ In order to see this, substitute $u = \bar{u}$ and $v_j = 0$ for each $j \in  [\bar{\omega}]$ in \eqref{eq:ineq separation theorem} to obtain

$$\left\|\sum\limits_{j=1}^{\bar{\omega}} \nabla f^{a_j}(\bar{x})\mu_j \right\|^2 \leq 0.$$ Hence, it can only be $\bar{u}=0,$ and statement \eqref{eq: stationarity lower def} is true.
   
\end{proof}

\section{Descent Method and its Convergence Analysis}\label{sec:algo algo}

Now, we present the solution approach. It is clearly based on the result shown in Lemma \ref{lem: equivalent vector problem lower}. At every iteration, an element $a$ in the partition set of the current iterate point is selected, and then a descent direction for \eqref{eq: associated vector opt problem lower} will be found using ideas from \cite{ChuongYao2012, DrummondSvaiter2005}. However, one must be careful with the selection process of the element $a$ in order to guarantee convergence. Thus, we propose a specific way to achieve this. After the descent direction is determined, we follow a classical backtracking procedure of Armijo type to determine a suitable step size, and we update the iterate in the desired direction. Formally, the method is the following:

\begin{algorithm}{
\caption{{Descent Method in Set Optimization}}\label{alg:setopt desc}
\textbf{Step 0.}  Choose $x_0 \in \R^n, \beta, \nu \in (0,1), $ and set $k:=0$.\\
\textbf{Step 1.} Compute
\begin{equation*}
M_k:= \Min(F(x_k),K), \;\;  P_k:= P_{x_k}, \;\; \omega_k:= |\Min(F(x_k),K)|.
\end{equation*}\\
\textbf{Step 2.} Find 
$$(a_k, u_k) \in \underset{(a,u)\in P_k \times \R^n}{\argmin}   \max_{ j \in \;[\omega_k]} \big\{\psi_e\big(\nabla f^{a_j}(x_k)^\top d\big)\big\} +\frac{1}{2}\|u\|^2. $$ \\
\textbf{Step 3.} If $u_k=0,$ Stop. Otherwise, go to Step 4.\\
\textbf{Step 4.} Compute

$$t_k:= \max_{q\in \N \cup \{0\}}\bigg\{\nu^q \mid \; \forall \; j\in  [\omega_k]: f^{a_{k,j}}(x_k +\nu^qu_k)\preceq f^{a_{k,j}}(x_k) +\beta \nu^q\nabla f^{a_{k,j}}(x_k)^\top u_k \bigg\}.$$\\
\textbf{Step 5.} Set $x_{k+1}:= x_k +t_k u_k, \; k:= k+1$ and go to Step 2.}
\end{algorithm}

\begin{Remark}\label{rem:algo extends ds}

Algorithm  \ref{alg:setopt desc} extends the approaches proposed in \cite{ChuongYao2012, DrummondSvaiter2005} for vector optimization optimization  problems to the case \eqref{eq: SP}. The main difference is that, in Step 2, the authors use the well known Hiriart-Urruty functional and the support of a so called generator of the dual cone instead of $\psi_e$, respectively. However, in in our framework, the functional $\psi_e$ is a particular case of those employed in the other methods, see \cite[Corollary 2]{bouzaquintanatammer2019}. Thus, the equivalence in the case of vector optimization problems of the three  algorithms is obtained. 
\end{Remark}

Now, we start the convergence analysis of Algorithm \ref{alg:setopt desc}. Our first lemma describes local properties of the active indexes.

\begin{Lemma}\label{lem: indexes}

Under our assumptions, there exists a neighborhood $U$ of $\bar{x}$ such that the following properties are satisfied (some  of them under additional conditions to be established below) for every $x \in U:$
\begin{enumerate}

\item $ I_0(x)\subseteq I_0(\bar{x}),$

\item $I(x)\subseteq I(\bar{x}),$ provided that  $\Min(F(\bar{x}),K)= \WMin(F(\bar{x}),K),$

\item  $\forall \; v\in \Min(F(\bar{x}),K): \Min \left(\{f^i(x)\}_{i \in I_v(\bar{x})},K\right) \subseteq \Min(F(x),K), $

\item For every $v_1,v_2 \in \Min(F(\bar{x}),K)$ with $v_1 \neq v_2 : $

$$\Min \left(\{f^i(x)\}_{i\in I_{v_1}(\bar{x})},K\right) \cap \Min\left(\{f^i(x)\}_{i\in I_{v_2}(\bar{x})},K\right)= \emptyset,$$

\item $ \omega(x)\geq \omega(\bar{x}).$ 
 
\end{enumerate}

\end{Lemma}

\begin{proof}
It  suffices to show the existence of the neighborhood $U$ for each item independently, as we could later take the intersection of them to satisfy all the properties.

$(i)$ Assume that this is not satisfied in any neighborhood $U$ of $\bar{x}.$ Then, we could find a sequence $\{x_k\}_{k\geq 1} \subset \R^n$ such that $x_k \to \bar{x}$ and 

\begin{equation}\label{eq:difference aux}
\forall \; k \in \N: I_0(x_k) \setminus I_0(\bar{x}) \neq \emptyset.
\end{equation} Because of the finite cardinality of all possible differences in \eqref{eq:difference aux}, we can assume without loss of generality that there exists a common $\bar{i}\in [p]$ such that

\begin{equation}\label{eq:common index aux}
\forall \; k\in \N: \bar{i} \in I_0(x_k) \setminus I_0(\bar{x}).
\end{equation} In particular, \eqref{eq:common index aux} implies that $\bar{i}\in I_0(x_k).$ Hence, we get 

$$\forall\; k \in \N,\;  i \in [p]:\;f^{i}(x_k) - f^{\bar{i}}(x_k) \in - \left(\R^m\setminus \Int K\right).$$  Since $\R^m\setminus \Int K$ is closed, taking the limit when $k \to +\infty$ we obtain 

$$\forall\; i \in [p]:\;f^i(\bar{x}) - f^{\bar{i}}(\bar{x}) \in - \left(\R^m\setminus \Int K\right).$$ Hence, we deduce that $f^{\bar{i}}(\bar{x}) \in \WMin(F(\bar{x}),K)$ and $\bar{i} \in I_0(\bar{x}),$ a contradiction to \eqref{eq:difference aux}. 

$(ii)$ Consider the same neighborhood $U$ on which statement $(i)$ holds. Note that, under the given assumption, we have $I_0(\bar{x}) = I(\bar{x}).$ This, together with statement $(i),$ implies:

$$\forall\; x\in U:\; I(x) \subseteq I_0(x) \subseteq I_0(\bar{x}) = I(\bar{x}).$$

$(iii)$ For this statement, it is also sufficient to show that the neighborhood $U$ can be chosen for any point in the set $ \Min(F(\bar{x}),K).$ Hence, fix  $v \in \Min(F(\bar{x}),K)$  and assume that there is no neighborhood $U$ of $\bar{x}$ on which the statement is satisfied. Then, we could find  sequences $\{x_k\}_{k\geq 1}\subset \R^n $ and $\{i_k\}_{k\geq 1} \subseteq I_v(\bar{x})$  such that $x_k \to \bar{x}$ and

\begin{equation}\label{eq:min/min1}
\forall \; k \in \N: f^{i_k}(x_k)  \in \Min(\{f^i(x_k)\}_{i \in I_v(\bar{x})},K) \setminus \Min(F(x_k),K).
\end{equation} Since $I_v(\bar{x})$ is finite, we deduce that there is only a finite number of different elements in the sequence $\{i_k\}.$ Hence, we can assume without loss of generality that there exists $\bar{i} \in I_v(\bar{x})$ such that $i_k = \bar{i}$ for every $k \in \N.$ Then, \eqref{eq:min/min1}, is equivalent to
\begin{equation}\label{eq:min/min2}
\forall \; k \in \N: f^{\bar{i}}(x_k)  \in \Min(\{f^i(x_k)\}_{i \in I_v(\bar{x})},K) \setminus \Min(F(x_k),K).
\end{equation} From \eqref{eq:min/min2},  we get in particular that $f^{\bar{i}}(x_k)  \notin \Min(F(x_k),K)$ for every $k \in \N.$ This, together with the domination property in Proposition \ref{prop:domination property} and the fact that the sets $I(x_k)$ are contained in the finite set $[p]$, allow us to obtain without loss of generality the existence of $\tilde{i} \in I(\bar{x})$ such that 

\begin{equation}\label{eq: lemma (ii)}
\forall\; k\in \N:\; f^{\tilde{i}}(x_k)\preceq f^{\bar{i}}(x_k), \; f^{\tilde{i}}(x_k) \neq f^{\bar{i}}(x_k).
\end{equation}

Now, taking the limit in \eqref{eq: lemma (ii)} when $k \to +\infty,$ we obtain $f^{\tilde{i}}(\bar{x})\preceq f^{\bar{i}}(\bar{x}) = v.$ Since $v$ is a minimal element of $F(\bar{x})$, it can only be $f^{\tilde{i}}(\bar{x})= v$ and, hence, $\tilde{i} \in I_v(\bar{x}).$ From this, the first inequality in \eqref{eq: lemma (ii)}, and the fact that $f^{\bar{i}}(x_k) \in \Min(\{f^i(x_k)\}_{i \in I_v(\bar{x})},K)$ for every $k \in \N,$ we get that $f^{\bar{i}}(x_k) = f^{\tilde{i}}(x_k)$ for all $k\in \N.$ This contradicts the second part of \eqref{eq: lemma (ii)}, and hence our statement is true.

$(iv)$ It follows directly from the continuity of the functionals $f^i, \; i \in [p].$

$(v)$ The statement is an immediate consequence of $(iii)$ and $(iv).$
\end{proof}

For the main convergence theorem of our method, we will need the notion of regularity  of a point for a set-valued mapping. 

\begin{Definition}\label{def:regular point}
We say that $\bar{x}$ is a regular point of $F$ if the following conditions are satisfied:

\begin{enumerate}
\item $\Min(F(\bar{x}),K)= \WMin(F(\bar{x}),K),$

\item the functional $\omega$ is constant in a neighborhood of $\bar{x}.$
\end{enumerate}
\end{Definition}
\begin{Remark}

Since we will analyze the stationarity of the regular limit points of the sequence generated by Algorithm \ref{alg:setopt desc}, the following points must be addressed: 

\begin{itemize}

\item Notice that, by definition, the regularity property of a point is independent of our optimality concept. Thus, by only knowing that a point is regular, we can not infer anything about whether it is optimal or not.

\item The concept of regularity seems to be linked to the complexity of comparing sets in a high dimensional space. For example, in case  $m=1$ or $p = 1,$ every point in $\R^n$ is regular for the set-valued mapping $F$. Indeed, in these cases, we have $\omega(x) = 1$ and

\[ 
\Min(F(x),K)= \WMin(F(x),K) = \left\{
\begin{array}{ll}
      \left\{\min\limits_{i \in \; [p]} f^i(x)\right\} & \textrm{if } m=1,  \\
      \{f^1(x)\} & \textrm{if } p=1\\
\end{array} 
\right. 
\] for all $x\in \R^n.$

\end{itemize}
\end{Remark}

A natural question  is whether regularity is a strong assumption to impose on a point. In that sense,  given the finite structure of the sets $F(x)$, the condition $(i)$ in Definition \ref{def:regular point} seems to be very reasonable. In fact, we would expect that, for most practical cases, this condition is fulfilled at almost every point. For condition $(ii)$, a formalized statement is derived in Proposition \ref{prop:dense subset} below.

\begin{Proposition} \label{prop:dense subset}

The set 

$$S: = \{x \in \R^n \mid  \omega \textrm{ is locally constant at } x\}$$ is open and dense in $\R^n.$
\end{Proposition}

\begin{proof}
$(i)$ The openness is trivial. Suppose now that $S$ is not dense in $\R^n.$ Then, $\R^n\setminus (\cl S)$ is nonempty and  open. Furthermore, since $\omega$ is bounded above, the real number 

$$p_0 := \max_{x \in \R^n\setminus (\cl S) } \omega(x)$$  is well defined. Consider the set 

$$A:= \left\{x \in \R^n \mid \omega(x)\leq p_0-\frac{1}{2}\right\}.$$ From Lemma \ref{lem: indexes} $(v)$, it follows that $\omega$ is lower semicontinuous. Hence, $A$ is closed as it is the sublevel set of a lower semicontinuous functional, see \cite[Lemma 1.7.2]{Schirotzek2007}. Consider now the set 

$$U:= \left(\R^n\setminus (\cl S)\right)\cap \left(\R^n \setminus A\right).$$ Then, $U$ is a nonempty open subset of $\R^n\setminus (\cl S).$ This, together with the definition of $A,$ gives us $\omega(x) = p_0$ for every $x \in U.$ However, this contradicts the fact that $\omega$ is not locally constant at any point of $\R^n\setminus (\cl S).$ Hence, $S$ is dense in $\R^n.$

\end{proof}

An essential property of regular points of a set-valued mapping is described in the next lemma.

\begin{Lemma}\label{prop:regularity property}

Suppose that $\bar{x}$ is a regular point of $F.$ Then, there exists a neighborhood $U$ of $\bar{x}$ such that the following properties  hold  for every $x \in U$:

\begin{enumerate}
\item $\omega(x) = \bar{\omega},$

\item there is an enumeration $\{w^x_1, \ldots ,w^x_{\bar{\omega}}\}$ of $\Min(F(x),K)$ such that 

$$\forall \; j \in [\bar{\omega}]: I_{w^x_j}(x) \subseteq I_{v^{\bar{x}}_j}(\bar{x}).$$ 
\end{enumerate}

In particular, without loss of generality, we have $P_x\subseteq P_{\bar{x}}$ for every $x \in U.$
   
\end{Lemma}

\begin{proof} Let $U$ be the neighborhood of $\bar{x}$ from Lemma \ref{lem: indexes}. Since $\bar{x}$ is a regular point of $F,$ we assume without loss of generality that $\omega$ is constant on $U.$ Hence, property $(i)$ is  fulfilled. Fix now $x \in U$ and consider the enumeration $\{v^{\bar{x}}_1,\ldots, v^{\bar{x}}_{\bar{\omega}}\}$ of $\Min(F(\bar{x}),K).$ Then, from properties $(iii)$ and $(iv)$ in Lemma \ref{lem: indexes} and the fact that $\omega(x)= \bar{\omega},$  we deduce that 

\begin{equation}\label{eq:card 1}
\forall\; j\in [\bar{\omega}]:\; \left|\Min\left(\{f^i(x)\}_{i\in I_{v^{\bar{x}}_j}(\bar{x})},K\right)\right| =1.
\end{equation} Next, for  $j \in [\bar{\omega}],$ we define $w^x_j$ as the unique element of the set $\Min\left(\{f^i(x)\}_{i\in I_{v^{\bar{x}}_j}(\bar{x})},K\right).$ Then, from \eqref{eq:card 1}, property $(iii)$ in Lemma \ref{lem: indexes} and the fact that $\omega$ is constant on $U,$ we obtain that $\{w^x_1,\ldots, w^x_{\bar{\omega}}\}$ is an enumeration of the set $\Min(F(x),K).$ 

It remains to show now that this enumeration satisfies $(ii).$ In order to see this, fix $j \in  [\bar{\omega}]$ and $\bar{i} \in I_{w^x_j}(x).$ Then, from the regularity of $\bar{x}$ and property $(ii)$ in Lemma \ref{lem: indexes}, we get that $I(x)\subseteq I(\bar{x}).$ In particular, this implies  $\bar{i} \in I(\bar{x}).$ From this and \eqref{eq:partition of I}, we have the existence of $j' \in  [\bar{\omega}]$ such that $\bar{i} \in I_{v^{\bar{x}}_{j'}}(\bar{x}).$ Hence, we deduce that 

\begin{equation}\label{eq:w in other minimal}
w^x_j = f^{\bar{i}}(x) \in \left\{f^i(x)\right\}_{i \in I_{v^{\bar{x}}_{j'}}(\bar{x})}.
\end{equation} Then, from \eqref{eq:card 1}, \eqref{eq:w in other minimal} and the definition of  $w^x_{j'},$ we find that $w^x_{j'}\preceq w^x_j.$ Moreover, because $w^x_{j'}, w^x_j \in \Min(F(x),K),$ it can only be $w^x_{j'}= w^x_j.$ Thus, it follows that $j = j',$ since $\{w^x_1,\ldots, w^x_{\bar{\omega}}\}$ is an enumeration of the set $\Min(F(x),K).$ This shows that $\bar{i} \in I_{v^{\bar{x}}_j}(\bar{x}),$ as desired.

\end{proof}

For the rest of the analysis we need to introduce the parametric family of functionals $\{\varphi_x\}_{x \in \R^n},$ whose elements  $\varphi_x: P_x\times \R^n \to \R$ are defined as follows: 

\begin{equation}\label{eq:varphi_x}
\forall\; a \in P_x,u \in \R^n: \varphi_x(a,u):=\max_{ j \in \;[\omega(x)]} \big\{\psi_e\big(\nabla f^{a_j}(x)^\top u\big)\big\} +\frac{1}{2}\|u\|^2.
\end{equation} It is easy to see that, for every $x \in \R^n$ and $a \in P_x,$ the functional $\varphi_x(a, \cdot)$ is strongly convex in $\R^n,$ that is, there exists a constant $\alpha >0$ such that  the inequality

$$ \varphi_x\left(a,tu + (1-t)u'\right) + \alpha t(1-t)\|u-u'\|^2 \leq t \varphi_x(a,u) + (1-t)\varphi_x\left(a,u'\right)$$ is satisfied for every $u, u'\in \R^n$ and $t \in [0,1].$  According to \cite[Lemma 3.9]{kanzow1999u}, the functional $\varphi_x(a,\cdot)$ attains its minimum over $\R^n,$ and this minimum is unique. In particular, we can check that 
\begin{equation}\label{eq:min varphi leq 0}
\forall \; x \in \R^n, a \in P_x : \min_{u \in \R^n} \varphi_x(a,u) \leq 0
\end{equation} and that, if $u_a \in \R^n$ is such that $\varphi_x(a,u_a) = \min\limits_{u \in \R^n} \varphi_x(a,u),$  then

\begin{equation}\label{eq:var = 0}
\varphi_x(a,u_a) = 0 \Longleftrightarrow u_a = 0.
\end{equation} Taking into account that $P_x$ is finite, we also obtain that $\varphi_x$ attains its minimum over the set  $P_x\times \R^n.$  Hence, we can consider the functional $\phi: \R^n \rightarrow \R$ given by 

\begin{equation}\label{eq:phi functional}
\phi(x):= \min_{(a,u)\in P_x \times \R^n} \varphi_x(a,u).
\end{equation} Then, because of \eqref{eq:min varphi leq 0}, it can be verified that 

\begin{equation}\label{eq:phi <=0}
\forall\; x \in  \R^n: \phi(x)\leq 0.
\end{equation}
Furthermore, if $(a,u) \in P_x\times \R^n$ is such that $\phi(x) = \varphi_x(a,u),$ it follows from \eqref{eq:var = 0} (see also \cite{DrummondSvaiter2005}) that

\begin{equation}\label{eq: phi = 0 equivalence}
\phi(x) = 0 \Longleftrightarrow u = 0.
\end{equation}

In the  following two propositions we  show that Algorithm \ref{alg:setopt desc} is well defined. We start by proving that, if Algorithm \ref{alg:setopt desc} stops in Step 3, a   stationary point was found.

\begin{Proposition}\label{prop: charact stat}
Consider the functionals $\varphi_{\bar{x}}$ and $\phi$ given in \eqref{eq:varphi_x} and \eqref{eq:phi functional}, respectively. Furthermore, let  $(\bar{a},\bar{u})\in P_{\bar{x}} \times \R^n$ be such that $\phi(\bar{x}) =  \varphi_{\bar{x}}(\bar{a},\bar{u}).$ Then, the following statements are equivalent:

\begin{enumerate}

\item $\bar{x}$ is a strongly    stationary point of \eqref{eq: SP},

\item $\phi(\bar{x})=0,$

\item $\bar{u}=0.$
\end{enumerate}

\end{Proposition}

\begin{proof}
The result will be a consequence of \cite[Proposition 2.2]{ChuongYao2012} where, using an Hiriart- Urruty functional, a similar statement is proved for vector optimization problems. Consider the cone $\tilde{K}$ given by \eqref{eq:tildeK} , the vector $\tilde{e}: = \begin{pmatrix}
e \\ \vdots \\e
\end{pmatrix}  \in \Int \tilde{K},$ and the scalarizing functional $\psi_{\tilde{e}}$ associated to $\tilde{e}$ and $\tilde{K},$ see Definition \ref{def:funcGW}. Then, for any  $v_1,\ldots,v_{\bar{\omega}} \in \R^m$ and  $v:= \begin{pmatrix}
v_1 \\ \vdots \\ v_{\bar{\omega}}
\end{pmatrix},$ we get

\begin{equation}\label{eq:new tammer func in product space}
\begin{aligned}
\psi_{\tilde{e}}(v)&= \min\{t \in \R \mid t\tilde{e} \in v + \tilde{K} \}\\
                   & = \min\{t \in \R \mid \forall\; j\in  [\bar{\omega}]: te \in v_j + K\}\\
                   & =  \max_{ j \in \;[\bar{\omega}]}\psi_e(v_j).
\end{aligned}
\end{equation} From \cite[Theorem 4]{bouzaquintanatammer2019}, we know that $\psi_{\tilde{e}}$ is an Hiriart-Urruty functional. Hence, for a fixed $a \in P_{\bar{x}}$ we can apply \cite[Proposition 2.2]{ChuongYao2012} to  \eqref{eq: associated vector opt problem lower} to obtain that

\begin{equation}\label{eq:stat equivalence}
\bar{x} \textrm{ is a stationary point of \eqref{eq: associated vector opt problem lower}}  \Longleftrightarrow \min\limits_{u \in \R^n} \;\psi_{\tilde{e}}\left(\nabla \tilde{f}^a(\bar{x})^\top u\right)  + \frac{1}{2}\|u\|^2 = 0
\end{equation} Thus, we deduce that

\begin{eqnarray*}
\bar{x} \textrm{ is strongly    stationary} & \overset{\textrm{ (Remark \ref{rem:weak stat is stat for vp}) }}{\Longleftrightarrow}   & \forall\; a\in P_{\bar{x}}: \bar{x} \textrm{ is stationary for \eqref{eq: associated vector opt problem lower}}\\
           &  \overset{ \eqref{eq:stat equivalence}}{\Longleftrightarrow} & \forall\; a\in P_{\bar{x}}: \min\limits_{u \in \R^n} \;\psi_{\tilde{e}}\left(\nabla \tilde{f}^a(\bar{x})^\top u\right)  + \frac{1}{2}\|u\|^2 = 0\\ 
           & \overset{( \eqref{eq:varphi_x} \; + \;\eqref{eq:new tammer func in product space})}{\Longleftrightarrow} & \forall\; a\in P_{\bar{x}}: \min\limits_{u \in \R^n}\varphi_{\bar{x}}(a,u) = 0\\
           & \Longleftrightarrow & \min\limits_{(a,u) \in P_{\bar{x}}\times \R^n}  \varphi_{\bar{x}}(a,u) = 0\\
           & \overset{\eqref{eq:phi functional}}{\Longleftrightarrow} & \phi(\bar{x}) = 0\\
           & \overset{\eqref{eq: phi = 0 equivalence}}{\Longleftrightarrow} &  \bar{u} = 0,
\end{eqnarray*} as desired.

\end{proof}

\begin{Remark}\label{rem:q stat equivalence}
A similar statement to the one in Proposition  \ref{prop: charact stat} can be made for stationary points of \eqref{eq: SP}. Indeed, for a set $Q \subseteq P_{\bar{x}},$ consider a point $\left(\bar{a}_Q,\bar{u}_Q\right) \in Q \times \R^n$ such that $\varphi_{\bar{x}}\left(\bar{a}_Q,\bar{u}_Q\right) = \min\limits_{(a,u) \in Q \times \R^n} \varphi_{\bar{x}}(a,u).$ Then, by replacing $P_{\bar{x}}$ by $Q$ in the proof of Proposition \ref{prop: charact stat}, we can show that the following statements are equivalent:

\begin{enumerate}

\item $\bar{x}$ is   stationary for \eqref{eq: SP} with respect to $Q,$
\item $\min\limits_{(a,u) \in Q \times \R^n} \varphi_{\bar{x}}(a,u) = 0,$
\item $\bar{u}_Q = 0.$
\end{enumerate}
\end{Remark}

Next, we show that the line search in Step 4 of Algorithm \ref{alg:setopt desc} terminates in finitely many steps.

\begin{Proposition}\label{prop:linesearch is well defined too}
Fix $\beta \in (0,1)$ and consider the functionals $\varphi_{\bar{x}}$ and $\phi $ given in \eqref{eq:varphi_x} and \eqref{eq:phi functional}  respectively. Furthermore,  let  $\left(\bar{a},\bar{u}\right)\in P_{\bar{x}} \times \R^n$ be such that $\phi(\bar{x}) =  \varphi_{\bar{x}}(\bar{a},\bar{u})$ and suppose that $\bar{x}$ is not a strongly   stationary point of \eqref{eq: SP}. The following assertions hold:
\begin{enumerate}
\item There exists $\tilde{t} > 0$ such that

\begin{equation*}
\forall\; t \in (0,\tilde{t}\;], j \in  [\bar{\omega}]: f^{\bar{a}_j}(\bar{x} +t\bar{u})\preceq f^{\bar{a}_j}(\bar{x}) +\beta t\nabla f^{\bar{a}_j}(\bar{x})^\top \bar{u}.
\end{equation*}
\item  Let $\tilde{t}$ be the parameter in statement $(i)$. Then,

\begin{equation*}
\forall \; t\in (0,\tilde{t}\;]: F(\bar{x} + t \bar{u}) \preceq^\ell \left\{f^{\bar{a}_j}(\bar{x}) + \beta t \nabla f^{\bar{a}_j}(\bar{x})^\top \bar{u}\right\}_{j \in  [\bar{\omega}]}   \prec^\ell F(\bar{x}),
\end{equation*} In particular, $\bar{u}$ is a descent direction of $F$ at $\bar{x}$ with respect to the preorder $\preceq^\ell.$
\end{enumerate}

\begin{proof}
$(i)$ Assume otherwise. Then, we could find  a sequence $\{t_k\}_{k\geq 1}$ and $\bar{j}\in   [\bar{\omega}]$ such that $t_k \to 0$ and

\begin{equation}\label{eq:ineqlinesearch}
\forall \; k \in \N: f^{\bar{a}_{\bar{j}}}(\bar{x} +t_k\bar{u}) - f^{\bar{a}_{\bar{j}}}(\bar{x}) - \beta t_k\nabla f^{\bar{a}_{\bar{j}}}(\bar{x})^\top \bar{u} \notin  -K.
\end{equation} As $(\R^m\setminus - K)\cup \{0\}$ is a cone, we can multiply \eqref{eq:ineqlinesearch} by $\frac{1}{t_k}$ for each $k \in \N$ to obtain

\begin{equation}\label{eq:ineqlinesearch fractional}
\forall \; k \in \N: \frac{f^{\bar{a}_{\bar{j}}}(\bar{x} +t_k\bar{u}) - f^{\bar{a}_{\bar{j}}}(\bar{x})}{t_k} - \beta \nabla f^{\bar{a}_{\bar{j}}}(\bar{x})^\top \bar{u} \notin  -K.
\end{equation} Taking now the limit in \eqref{eq:ineqlinesearch fractional} when $k \to +\infty$ we get

\begin{equation*}
(1-\beta)\nabla f^{\bar{a}_{\bar{j}}}(\bar{x})^\top \bar{u} \notin  - \Int K.
\end{equation*} Since $\beta \in (0,1),$ this is equivalent to 

\begin{equation}\label{eq:ineqlinesearch to contradict}
\nabla f^{\bar{a}_{\bar{j}}}(\bar{x})^\top \bar{u} \notin - \Int K.
\end{equation} On the other hand, since $\bar{x}$ is not strongly  stationary, we can apply Proposition \ref{prop: charact stat} to obtain that $\bar{u} \neq 0$ and that $\phi(\bar{x})<0.$  This implies that $\varphi_{\bar{x}}(\bar{a},\bar{u}) <0,$ and hence

$$\max\limits_{j \in [\bar{\omega}]} \left\{\psi_e \left(\nabla f^{\bar{a}_j}(\bar{x})^\top \bar{u}\right) \right\} < -\frac{1}{2}\|\bar{u}\|^2 <0.$$ From this, we deduce that 

\begin{equation*}
\forall \; j \in  [\bar{\omega}]: \psi_e \left(\nabla f^{\bar{a}_j}(\bar{x})^\top \bar{u}\right) < 0 
\end{equation*} and, by Proposition \ref{prop: properties of tammer function} $\ref{item:representability of scalarization}$,

\begin{equation}\label{eq:nabla f -intK}
\forall \; j \in  [\bar{\omega}]: \nabla f^{\bar{a}_j}(\bar{x})^\top \bar{u} \in - \Int K.
\end{equation}
However, this is a contradiction to \eqref{eq:ineqlinesearch to contradict}, and hence the statement is proven.

$(ii)$ From \eqref{eq:nabla f -intK}, we know that 

\begin{equation}\label{eq:ineq aux armijo}
\forall \; j \in  [\bar{\omega}], t \in (0,\tilde{t}\;] :  f^{\bar{a}_j}(\bar{x}) + \beta t\nabla f^{\bar{a}_j}(\bar{x})^\top \bar{u} \prec  f^{\bar{a}_j}(\bar{x}) .
\end{equation} Then, it follows that

\begin{eqnarray*}
\forall\; t \in (0,\bar{t\;]}:  F(\bar{x}) & \overset{(\textrm{Proposition } \ref{prop:domination property})}{\subset} & \left\{ f^{\bar{a}_1}(\bar{x}), \ldots, f^{\bar{a}_{\bar{\omega}}}(\bar{x})    \right\} + K\\
           & \overset{\eqref{eq:ineq aux armijo}}{\subset}  & \left\{\nabla f^{\bar{a}_j}(\bar{x}) + \beta t\nabla f^{\bar{a}_j}(\bar{x})^\top \bar{u}\right\}_{j \in  [\bar{\omega}]} + \Int K\\
           & \overset{(\textrm{Statement } (i))}{\subseteq} & \left\{ f^{\bar{a}_j}(\bar{x}+ t \bar{a}_1), \ldots ,  f^{\bar{a}_j}(\bar{x}+ t \bar{a}_{\bar{\omega}})\right\}_{j \in  [\bar{\omega}]} + K + \Int K\\
           & \subseteq & F(\bar{x} + t \bar{u}) + \Int K,
\end{eqnarray*} as desired.
\end{proof}

\end{Proposition}

We are now ready to establish the convergence  of Algorithm \ref{alg:setopt desc}.

\begin{Theorem}\label{thm:algo convergence}
Suppose that Algorithm \ref{alg:setopt desc} generates an infinite sequence for which $\bar{x}$ is an accumulation point. Furthermore, assume that $\bar{x}$ is regular for $F.$ Then, $\bar{x}$ is   a stationary point of \eqref{eq: SP}. If in addition $|P_{\bar{x}}| = 1$,  then   $\bar{x}$ is a strongly stationary point of \eqref{eq: SP}.

\end{Theorem}
\begin{proof} Consider the functional $\zeta: \mathscr{P}(\R^m) \rightarrow \R \cup \{- \infty\}$ defined as 

$$\forall \;A \in \mathscr{P}(\R^m): \; \zeta(A):= \inf_{y \in A} \psi_e(y).$$ The proof will be divided in several steps:

\begin{flushleft}
\underline{\textbf{Step 1}}: We show the following result:

\begin{equation}\label{eq: descent lemma}
\forall\; k\in \N\cup\{0\}: \; (\zeta \circ F)(x_{k+1})\leq (\zeta \circ F)(x_k) + \beta t_k \left[\phi(x_k)- \frac{1}{2}\|u_k\|^2 \right].
\end{equation}
\end{flushleft}

Indeed, because of the monotonicity property of $\psi_e$ in Proposition \ref{prop: properties of tammer function} $\ref{item:monotonicity of scalarization}$, the functional $\zeta$ is monotone with respect to the preorder $\preceq^\ell$, that is,  $$\forall\; A,B \in \mathscr{P}(\R^m): A\preceq^\ell B \Longrightarrow \zeta(A)\leq \zeta(B).$$ On the other hand, from Proposition \ref{prop:linesearch is well defined too} $(ii)$, we deduce that 

$$\forall\; k \in \N \cup\{0\}: F(x_k+t_k u_k)\preceq^\ell \left\{f^{a_{k,j}}(x_k)+ \beta t_k \nabla f^{a_{k,j}}(x_k)^\top u_k \right\}_{i \in  [\omega_k]}.$$ Hence, using the monotonicity of $\zeta,$ we obtain for any $k\in \N\cup \{0\}:$ 

\begin{eqnarray*}
(\zeta \circ F)(x_{k+1}) & \leq & \min_{ j \in \; [\omega_k]}\left\{\psi_e \left(f^{a_{k,j}}(x_k)+ \beta t_k \nabla f^{a_{k,j}}(x_k)^\top u_k\right)\right\}\\
                          & \overset{(\textrm{Proposition \ref{prop: properties of tammer function} } \ref{item:sublinearity and lipschitz of scalarization} ) }{\leq} & \min_{ j \in \; [\omega_k]}\left\{\psi_e \left(f^{a_{k,j}}(x_k) \right) + \beta t_k \psi_e \left(\nabla f^{a_{k,j}}(x_k)^\top u_k\right)\right\}\\
                          & \leq & \min_{ j \in \; [\omega_k]}\left\{\psi_e \left(f^{a_{k,j}}(x_k) \right) + \beta t_k \max_{j' \in \; [\omega_k]} \left\{\psi_e \left(\nabla f^{a_{k,j'}}(x_k)^\top u_k  \right)\right\}\right\}\\ 
                          & = &  (\zeta \circ F)(x_k) + \beta t_k \max_{ j \in \;[\omega_k]} \left\{\psi_e \left(\nabla f^{a_{k,j}}(x_k)^\top u_k  \right)\right\}.
\end{eqnarray*} The above inequality, together with the definition of $\phi$ in \eqref{eq:phi functional}, implies \eqref{eq: descent lemma}.

 On the other hand, since $\bar{x}$ is an accumulation point of the sequence $\{x_k\}_{k\geq 0}$, we can find a subsequence $\mathcal{K}$ in $\N$ such that $x_k\overset{\mathcal{K}}{ \to} \bar{x}.$

\begin{flushleft}
\underline{\textbf{Step 2}}: The following inequality holds

\begin{equation}\label{eq: lower bound of F(x_k) is F(xbar)}
\forall\; k \in  \N \cup \{0\}:\; F(\bar{x})\preceq^\ell F(x_k).
\end{equation}
\end{flushleft}

Indeed, from Proposition \ref{prop:linesearch is well defined too} $(ii),$ we can guarantee that the sequence $\{F(x_k)\}_{k\geq 0}$ is decreasing with respect to the preorder $\preceq^\ell$, that is, 

\begin{equation}\label{eq:l-decreasing sequence}
\forall\; k\in \N \cup \{0\}: F(x_{k+1})\preceq^\ell F(x_k).
\end{equation} Fix now $k \in \N,$ and $i \in  [p].$ Then, according to \eqref{eq:l-decreasing sequence}, we have 

\begin{equation}\label{eq:l-decr aux}
\forall\; k' \in \mathcal{K},k' \geq k, \exists \;i_{k'} \in [p]: f^{i_{k'}}(x_{k'})\preceq f^i(x_k).
\end{equation} Since there are only a finite number of possible values for $i_{k'},$ we assume without loss of generality that there is $\bar{i} \in [p]$ such that $i_{k'} = \bar{i}$ for every $k' \in \mathcal{K}, k' \geq k.$ Hence, \eqref{eq:l-decr aux} is equivalent to

\begin{equation}\label{eq:l-decr aux 2}
\forall\; k' \in \mathcal{K},k' \geq k: f^{\bar{i}}(x_{k'}) - f^i(x_k) \in -K.
\end{equation} Taking the limit now in \eqref{eq:l-decr aux 2} when $k' \overset{\mathcal{K}}{\to} + \infty,$ we find that $$ f^i(x_k) \in f^{\bar{i}}(\bar{x}) + K.$$ Since $i$ was chosen arbitrarily in $[p],$ this implies the statement.

\begin{flushleft}
\underline{\textbf{Step 3}}: We prove that the sequence $\{u_k\}_{k\in \mathcal{K}}$ is bounded.
\end{flushleft}

In order to see this, note that, since $x_k$ is not a stationary point, we have by Proposition \ref{prop: charact stat} that $\phi(x_k) < 0$ for every $k\in \N\cup\{0\}.$ By the definition of $a_k$ and $u_k,$ we then have 

\begin{equation}\label{eq:var<0}
\forall \; k\in \N \cup\{0\}: \varphi_{x_k}(a_k,u_k)<0.
\end{equation} Let $\rho$ be the Lipschitz constant of $\psi_e$ from Proposition \ref{prop: properties of tammer function} $\ref{item:sublinearity and lipschitz of scalarization}.$ Then, we deduce that

\begin{eqnarray*}
\forall \; k\in \N \cup\{0\}:\|u_k\|^2 & \overset{(\eqref{eq:var<0}\; +\; \eqref{eq:varphi_x}) }{<}  & -2\max\limits_{j \in \; [\omega_k]}\left\{\psi_e\left(\nabla f^{a_{k,j}}(\bar{x})^\top u_k \right)  \right\}\\
          & = & 2\max\limits_{j \in \; [\omega_k]}\left\{\left|\psi_e\left(\nabla f^{a_{k,j}}(\bar{x})^\top u_k \right)\right|  \right\}\\
          & \overset{(\textrm{Proposition } \ref{prop: properties of tammer function}\; \ref{item:sublinearity and lipschitz of scalarization})}{\leq} & 2 \rho \max\limits_{j \in [\omega_k]}\left\{ \left\|\nabla f^{a_{k,j}}(\bar{x})^\top u_k \right\|\right\}\\
          & \leq & 2 \rho \|u_k\| \max\limits_{j \in \; [\omega_k] } \left\{ \|\nabla f^{a_{k,j}}(x_k)\|\right\}.
\end{eqnarray*} Hence,

\begin{equation}\label{eq:uk bounded xk}
\forall \; k\in \N \cup\{0\}: \|u_k\| \leq 2 \rho  \max\limits_{j \in \; [\omega_k] } \left\{ \|\nabla f^{a_{k,j}}(x_k)\|\right\}.
\end{equation} Since $\{x_k\}_{k\in \mathcal{K}}$ is bounded, the statement follows from \eqref{eq:uk bounded xk}.

\begin{flushleft}
\underline{\textbf{Step 4}}: We show that $\bar{x}$ is stationary. 
\end{flushleft}

Fix $\kappa \in \N.$ Then, it follows from \eqref{eq: descent lemma} that

\begin{equation}
\forall\; k\in \N: \; -\beta t_k \max_{ j \in \; [\omega_k]} \left\{\psi_e \left(\nabla f^{a_{k,j}}(x_k)^\top u_k  \right)\right\} \leq (\zeta \circ F)(x_k)- (\zeta \circ F)(x_{k+1}).
\end{equation} Adding this inequality for $k= 0,\ldots, \kappa,$ we obtain 

\begin{equation}\label{eq:sum bounded zeta}
-\beta\sum_{k=0}^\kappa t_k \max_{ j \in \; [\omega_k]} \left\{\psi_e \left(\nabla f^{a_{k,j}}(x_k)^\top u_k  \right)\right\} \leq (\zeta \circ F)(x^0)- (\zeta \circ F)(x^{\kappa+1}).
\end{equation}

On the other hand, similarly to \eqref{eq:nabla f -intK} in the proof of Proposition \ref{prop:linesearch is well defined too} $(i),$  we obtain that

\begin{equation}\label{eq:all in -IntK}
\forall \; k \in \N\cup\{0\},j \in  [\omega_k]: \nabla f^{a_{k,j}}(x_k)^\top u_k \in - \Int K.
\end{equation} In particular, applying Proposition \ref{prop: properties of tammer function} $\ref{item:representability of scalarization}$ in \eqref{eq:all in -IntK}, we find that

\begin{equation}\label{eq:all <0 scalarized}
\forall \; k \in \N\cup\{0\},j \in  [\omega_k]: \psi_e\left(\nabla f^{a_{k,j}}(x_k)^\top u_k \right)< 0.
\end{equation} We then have

\begin{eqnarray*}
0  \overset{\eqref{eq:all <0 scalarized}}{<}  -\sum\limits_{k=0}^\kappa t_k \max\limits_{ j \in \; [\omega_k]} \left\{\psi_e \left(\nabla f^{a_{k,j}}(x_k)^\top u_k  \right)\right\}    & \overset{\eqref{eq:sum bounded zeta}}{\leq} & \frac{(\zeta \circ F)(x^0)- (\zeta \circ F)(x^{\kappa+1})}{\beta}\\
					 & \overset{(\zeta\textrm{ monotone + }\eqref{eq: lower bound of F(x_k) is F(xbar)})}{\leq}  &\frac{(\zeta \circ F)(x^0)- (\zeta \circ F)(\bar{x})}{\beta}.
\end{eqnarray*} Taking now the limit in the previous inequality when $\kappa \to +\infty,$ we deduce that 

$$0 \leq -\sum_{k=0}^{\infty} t_k \max_{ j \in \; [\omega_k]} \left\{\psi_e \left(\nabla f^{a_{k,j}}(x_k)^\top u_k  \right)\right\} < +\infty.$$ In particular, this implies 

\begin{equation}\label{eq: t x optimality goes to 0}
\lim_{k\to \infty}t_k \max_{ j \in \; [\omega_k]} \left\{\psi_e \left(\nabla f^{a_{k,j}}(x_k)^\top u_k  \right)\right\} =0.
\end{equation}

Since there are only a finite number of subsets of $[p]$ and $\bar{x}$ is regular for $F,$ we can  apply Lemma \ref{prop:regularity property} to obtain, without loss of generality, the existence of  $Q \subseteq P_{\bar{x}}$ and $\bar{a} \in Q$ such that

\begin{equation}\label{eq: all index constant}
\forall \; k\in \mathcal{K}:\; \omega_{k} = \bar{\omega},\;P_{x_{k}} = Q,\;a_{k} = \bar{a}.
\end{equation} Furthermore, since the sequences $\{t_k\}_{k\geq 1}, \{u_{k}\}_{k\in \mathcal{K}}$ are bounded, we can also assume without loss of generality the  existence of $\bar{t} \in \R, \bar{u} \in \R^n$ such that 

\begin{equation}\label{eq: t,d limits}
 t_k \overset{\mathcal{K}}{\to} \bar{t},\;\; u_k \overset{\mathcal{K}}{\to} \bar{u}.
\end{equation} The rest of the proof is devoted to show that $\bar{x}$ is a    stationary point with respect to $Q.$ First, observe that by \eqref{eq: all index constant} and the definition of $a_k,$ we have 

$$\forall\; a\in Q,\; k\in \mathcal{K}, \; u\in \R^n: \phi(x_k)=\varphi_{x_k}(\bar{a}, u_k)\leq \varphi_{x_k}(a,u).$$ Then, taking into account that $\omega_k= \bar{\omega}$ in \eqref{eq: all index constant}, we can take the limit when $k \overset{ \mathcal{K}}{ \to} +\infty$ in the above expression to obtain

$$\forall\; a\in Q,\; u\in \R^n: \varphi_{\bar{x}}(\bar{a}, \bar{u})\leq \varphi_{\bar{x}}(a,u).$$ Equivalently, we have

\begin{equation}\label{eq: argmin of limit}
(\bar{a},\bar{u}) \in \argmin_{(a,u)\in Q\times \R^n} \varphi_{\bar{x}}(a,u).
\end{equation} Next, we analyze two cases:

\begin{flushleft}
\textbf{Case 1}: $\bar{t}>0.$
\end{flushleft}

According to \eqref{eq: t x optimality goes to 0} and \eqref{eq: all index constant}, we have in this case 

\begin{equation}\label{eq: optimality to 0}
\lim_{k \overset{ \mathcal{K}}{\to} + \infty} \max_{j \in [\bar{\omega}]} \left\{\psi_e \left(\nabla f^{\bar{a}_j}(x_k)^\top u_k  \right)\right\}=0.
\end{equation} Then, it follows that

\begin{eqnarray*}
0 & \leq & \frac{1}{2} \|\bar{u}\|^2\\
  &  \overset{ (\eqref{eq: all index constant}  \; + \;\eqref{eq: t,d limits}\;+ \;\eqref{eq: optimality to 0}) }{=}    & \lim_{k \overset{ \mathcal{K}}{\to} + \infty} \max_{ j \in \;[\bar{\omega}]} \left\{\psi_e \left(\nabla f^{\bar{a}_j}(x_k)^\top u_k  \right)\right\} + \frac{1}{2} \|u_k\|^2\\
  &  =  & \lim_{k \overset{ \mathcal{K}}{\to} + \infty} \phi(x_k)\\
  & \overset{\eqref{eq:phi <=0}}{\leq} & 0,
\end{eqnarray*} from which we deduce $\bar{u} = 0.$ This, together with \eqref{eq: argmin of limit} and Remark \ref{rem:q stat equivalence}, imply that $\bar{x}$ is a    stationary point with respect to $Q.$ 

\begin{flushleft}
\textbf{Case 2}: $\bar{t}=0.$
\end{flushleft}

Fix an arbitrary $\kappa \in \N.$ Since $t_k \overset{\mathcal{K}}{\to} 0,$ for $k \in \mathcal{K}$ large enough $\nu^{\kappa}$ does not satisfy Armijo's line search criteria in Step 4 of Algorithm \ref{alg:setopt desc}. By \eqref{eq: all index constant} and the  finiteness of $\bar{\omega},$ we can assume without loss of generality the existence of $\bar{j} \in [\bar{\omega}]$ such that 

$$\forall \; k\in \mathcal{K}:\; f^{\bar{a}_{\bar{j}}}(x_k+ \nu^{\kappa}u_k) \npreceq f^{\bar{a}_{\bar{j}}}(x_k) +\beta \nu^{\kappa}\nabla f^{\bar{a}_{\bar{j}}}(x_k)^\top u_k.$$ From this, it follows that

$$\forall \; k\in \mathcal{K}:\;\frac{f^{\bar{a}_{\bar{j}}}(x_k+ \nu^{\kappa}u_k) - f^{\bar{a}_{\bar{j}}}(x_k)}{\nu^{\kappa}} -\beta \nabla f^{\bar{a}_{\bar{j}}}(x_k)^\top u_k \notin -K.$$ Now, taking the limit when $k \overset{\mathcal{K}}{\rightarrow} +\infty,$ we obtain

$$\frac{f^{\bar{a}_{\bar{j}}}(\bar{x}+ \nu^{\kappa}\bar{u}) - f^{\bar{a}_{\bar{j}}}(\bar{x})}{\nu^{\kappa}} -\beta \nabla f^{\bar{a}_{\bar{j}}}(\bar{x})^\top\bar{u} \notin - \Int K.$$ Next, taking the limit when $\kappa \rightarrow +\infty,$ we get

$$(1-\beta)\nabla f^{\bar{a}_{\bar{j}}}(\bar{x})^\top\bar{u}\notin - \Int K.$$ Since $\beta \in (0,1),$ we deduce that $\nabla f^{\bar{a}_{\bar{j}}}(\bar{x})^\top\bar{u}\notin - \Int K$ and, according to Proposition \ref{prop: properties of tammer function} $\ref{item:representability of scalarization},$ this is equivalent to 

\begin{equation}\label{eq:psi nabla geq0}
\psi_e(\nabla f^{\bar{a}_{\bar{j}}}(\bar{x})^\top\bar{u})\geq 0.
\end{equation} Finally, we find that

\begin{equation*}
0 \overset{\eqref{eq:psi nabla geq0}}{\leq } \psi_e(\nabla f^{\bar{a}_{\bar{j}}}(\bar{x})^\top\bar{u}) \leq \varphi_{\bar{x}}(\bar{a},\bar{u}) \overset{\eqref{eq: argmin of limit}}{=} \min_{(a,u)\in Q\times \R^n} \varphi_{\bar{x}}(a,u) \overset{\eqref{eq:min varphi leq 0}}{\leq} 0,
\end{equation*} which implies 

\begin{equation}\label{eq:min = 00}
\min_{(a,u)\in Q\times \R^n} \varphi_{\bar{x}}(a,u) = 0.
\end{equation} The   stationarity of $\bar{x}$ follows then from \eqref{eq:min = 00} and Remark \ref{rem:q stat equivalence}. The proof is complete.

\end{proof}

\section{Implementation and Numerical Illustrations}\label{sec:numerical algo}

In this section, we report some preliminary numerical experience with the proposed method. Algorithm \ref{alg:setopt desc} was implemented in Python 3 and the experiments were done in a PC with an Intel(R) Core(TM) i5-4200U CPU processor and  4.0 GB of RAM. We describe below some details of the implementation and the experiments:

\begin{itemize}
\item We considered instances of problem \eqref{eq: SP} only for the case in which $K$ is the standard ordering cone, that is, $ K = \R_+^m.$ In addition, we choose the parameter $e \in \Int K$ for the scalarizing functional $\psi_e$ as $e = (1,\ldots,1)^{\top}.$

\item  The parameters $\beta$ and $\nu$ for the line search in Step 4 of the method were chosen as $\beta = 0.0001,\; \nu= 0.500.$

\item The stopping criteria employed was that  $\|u_k\|< 0.0001,$ or a maximum number of $200$ iterations  was reached.

\item For finding the set $\Min(F(x_k),K)$ at  the $k^{th}$-  iteration in Step 1 of the algorithm, we implemented the method developed by G\"{u}nther and Popovici in \cite{guntherpopovici2018}. This procedure requires a strongly monotone functional $\psi : \R^m \rightarrow \R$ with respect to the partial order $\preceq$ for a so called presorting phase. In our implementation, we used $\psi$ defined as follows:

$$\forall \; v\in \R^m: \psi(v) = \sum_{i = 1}^m v_i.$$ 

The other possibility  for finding the set $\Min(F(x_k),K)$ would have been to use the method introduced by Jahn in \cite{jahn2006subdiv,Jahn2011,jahnrathje2006}  with ideas from \cite{younes1993}. However, as mentioned in the introduction, the first approach has better computational complexity. Thus, the algorithm proposed in \cite{guntherpopovici2018} was a clear choice.

\item At the $k^{th}$-  iteration in Step 2 of the algorithm, we worked with the modelling language CVXPY 1.0 \cite{cvxpy_rewriting, cvxpy} for the solution of the problem 

$$ \underset{(a,u)\in P_k \times \R^n}{\min}    \varphi_{x_k}(a,u).$$  Since the variable $a$ is constrained to be in the discrete set $P_k,$ we proceeded as follows: using the solver ECOS \cite{Domahidi2013ecos} within CVXPY, we compute for every $a \in P_k$ the unique solution $u_a$ of the strongly convex problem 

$$ \underset{u\in \R^n}{\min}    \varphi_{x_k}(a,u).$$ Then, we set 

$$(a_k,u_k) = \argmin\limits_{a \in P_k} \varphi_{x_k}(a,u_a).$$ 

\item For each test instance considered in the experimental part,  we generated initial points randomly on a specific set and run the algorithm. We define as solved those experiments in which the algorithm stopped because $\|u_k\|< 0.0001$, and declared that a strongly  stationary point was found. For a given experiment, its final error is the value of $\|u_k\|$ at the last iteration. The following variables are collected for each test instance:
\begin{itemize}
\item \textbf{Solved}: this value indicates the number of initial points for which the problem was solved.

\item \textbf{Iterations}: this is a 3-tuple (min, mean, max) that indicates the minimum, the  mean, and the maximum  of the number of iterations in those instances reported as solved.

\item \textbf{Mean CPU Time}: Mean of the CPU time(in seconds) among the solved cases.

\end{itemize} Furthermore, for clarity, 
all the numerical values will be displayed for  up to  four decimal places.
\end{itemize}

Now, we proceed to the different instances on which our algorithm was tested. Our first test instance can be seen as a continuous version of an example in \cite{guntherkobispopovici2019}.
\begin{Test Instance} \label{test1}
We consider  $F:\mathbb{R} \rightrightarrows \mathbb{R}^2$ defined as 

$$F(x) := \left\{f^1(x), \ldots, f^5(x)\right\}$$ where, for $i \in [5],$  $f^i: \R \rightarrow \R^2$ is given as

$$f^i(x) := \begin{pmatrix}
x \\ \frac{x}{2}\sin(x) 
\end{pmatrix} + \sin^2(x)\left[ \frac{(i-1)}{4 }\begin{pmatrix}
1 \\ -1
\end{pmatrix}+  \left(1- \frac{(i-1)}{4}\right)\begin{pmatrix}
-1 \\ 1
\end{pmatrix}    \right].
$$ 
\end{Test Instance} The objective values in this case are discretized segments moving around a curve and being contracted (dilated) by a factor dependent on the argument. We generated $100$ initial points $x_0$ randomly on the interval $[-5\pi,5\pi]$ and run our algorithm. Some of the metrics are collected in Table \ref{tab1}. As we can see, in this case all the runs terminated finding a strongly  stationary point. Moreover, we observed that for this problem not too many iterations were needed.

\begin{table}[h]
\centering
\begin{tabular}{ |c|c|c|  }
 \hline
 \multicolumn{3}{|c|}{ \textbf{Test Instance \ref{test1}}} \\
 \hline
 \textbf{Solved} &  \textbf{Iterations}  & \textbf{Mean CPU Time} \\
 \hline
 100   &  (0, 13.97, 71)   & 0.1872   \\
 \hline
\end{tabular}

\caption{Performance of Algorithm \ref{alg:setopt desc} in Test Instance \ref{test1}}
\label{tab1} 
\end{table}
In Figure \ref{fig1}, the sequence $\left\{F(x_k)\right\}_{k \in \{0\}\cup [7]}$ generated  by Algorithm \ref{alg:setopt desc} for a selected starting point is shown. In this case, strong stationarity was declared after $7$ iterations. The traces of the curves $f^i$ for $i \in [5]$ are displayed, with arrows indicating their direction of movement. Moreover, the sets $F(x_0)$ and $F(x_{7})$ are represented by black and red points respectively, and the elements of the sets  $F(x_k)$  with $k \in [6]$  are in gray color. The improvements of the objective values after every iteration are clearly observed. 

\begin{figure}[ht]
\includegraphics[scale=0.7]{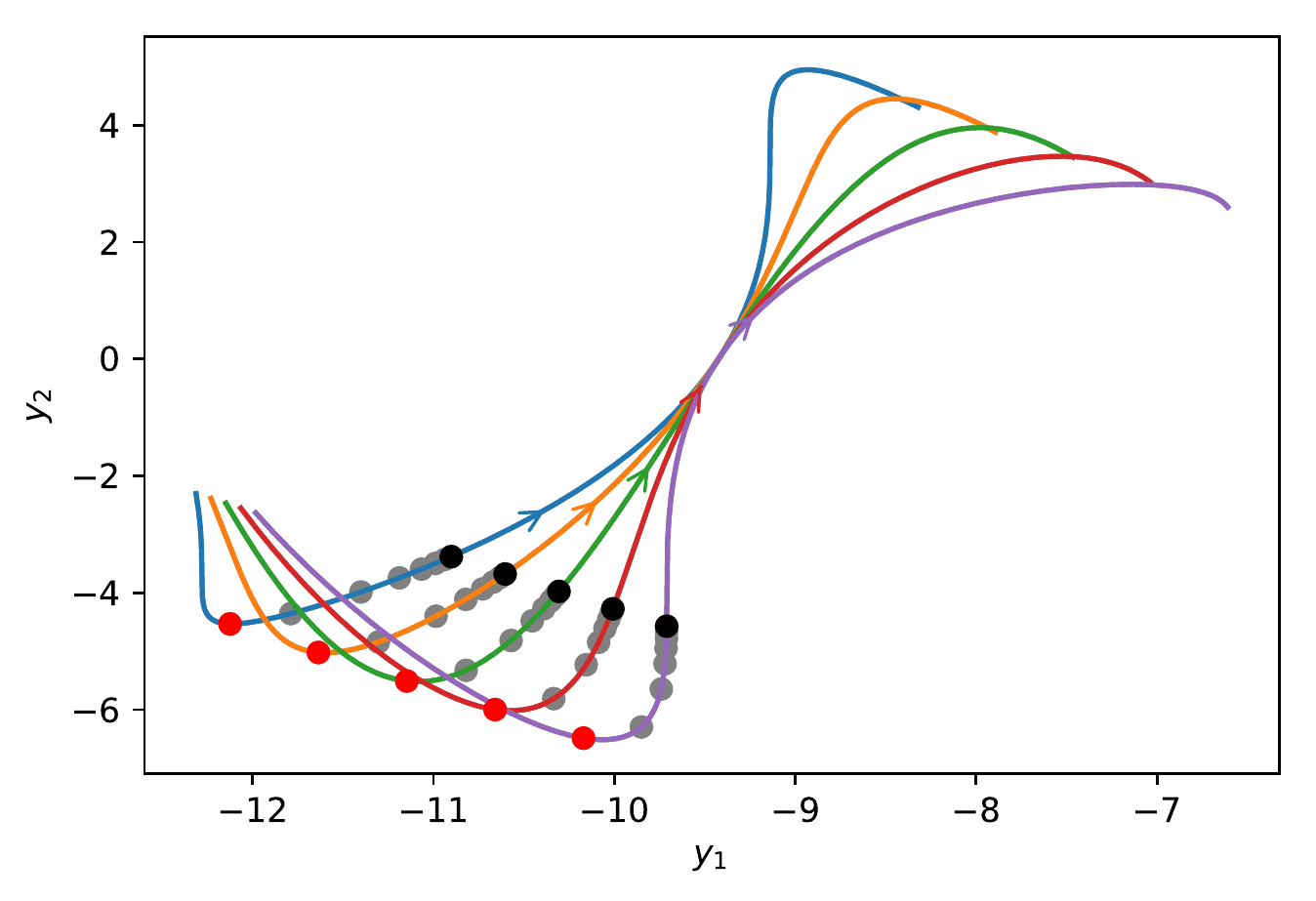}
\centering
\caption{Sequence generated  in the image space by Algorithm \ref{alg:setopt desc} for a selected starting point in Test Instance \ref{test1}}
\label{fig1}
\end{figure}

\begin{Test Instance}\label{test2}
  In this example, we start by taking a uniform partition $\mathcal{U}_1$ of 10 points of the interval $[-1,1]$ that is, 

$$\mathcal{U}_1= \{-1,  -0.7778, -0.5556, -0.3333, -0.1111,  0.1111,  0.3333,  0.5556,   0.7778,  1  \}.$$ Then, the set $\;\mathcal{U}: = \mathcal{U}_1\times \mathcal{U}_1$  is a mesh of 100 points of the square $[-1,1]\times [-1,1].$ Let $\{u_1,\ldots, u_{100}\}$ be an enumeration of $\mathcal{U}$ and consider the points 

$$l_1:= \begin{pmatrix}
0\\0
\end{pmatrix}, \; l_2:= \begin{pmatrix}
8\\0
\end{pmatrix},\; l_3:= \begin{pmatrix}
0\\8
\end{pmatrix}.$$  We define, for $i \in [100],$ the functional $f^i:\R^2 \rightarrow \R^3$ as 

$$f^i(x) := \frac{1}{2}\begin{pmatrix} \|x- l_1-u_i\|^2 \\ \|x- l_2-u_i\|^2 \\ \|x- l_3 -u_i\|^2 \end{pmatrix}.$$ Finally, the set-valued mapping $F: \R^2 \rightrightarrows \R^3$ is defined by 

$$F(x) := \left\{f^1(x), \ldots, f^{100}(x)\right\}.$$ 

\end{Test Instance}

Note that problem \eqref{eq: SP} corresponds in this case to the robust counterpart of a vector location problem under uncertainty \cite{IdeKobisKuroiwa2014}, where $\mathcal{U}$ represents the uncertainty set on the location facilities $l_1,l_2,l_3.$ Furthermore, with the aid of Theorem \ref{thm:oc sopt finite}, it is possible to show that a point $\bar{x}$ is a local weakly minimal solution of \eqref{eq: SP}  if and only if 

\begin{equation*}\label{eq:t2 charact}
\bar{x} \in \conv \left\{l_j + u_i \mid  (i,j) \in I(\bar{x})\times \{1,2,3\} \right\}.
\end{equation*} Thus, in particular, the local weakly minimal solutions lie on the set 

\begin{equation}\label{eq:t2 contained in}
C:=\conv\left((l_1+ \mathcal{U})\cup (l_2+ \mathcal{U}) \cup (l_3+ \mathcal{U}) \right).
\end{equation}

In this test instance, $100$ initial points $x_0$  were generated in the square $[-50,50]\times [-50,50],$ and Algorithm \ref{alg:setopt desc} was ran in each case. A summary of the results are presented in Table \ref{tab2}. Again, for any initial point the sequence generated by the algorithm stopped with a local solution to our problem. Perhaps the most noticeable parameter recorded in this case is the number of iterations required to declare the solution. Indeed, in most cases, only $1$ iteration was enough, even when the starting point was far away from the locations $l_1, l_2, l_3.$

\begin{table}[h]
\centering
\begin{tabular}{ |c|c|c|  }
 \hline
 \multicolumn{3}{|c|}{\textbf{Test Instance \ref{test2}}} \\
 \hline
 \textbf{Solved} & \textbf{Iterations}  & \textbf{Mean CPU Time} \\
 \hline
 100 & (0,1.32,2)    & 0.0637 \\
 \hline
\end{tabular}

\caption{Performance of Algorithm \ref{alg:setopt desc} in Test Instance \ref{test2}} 
\label{tab2} 
\end{table}

In Figure \ref{fig2}, the set of solutions found in this experiment are shown in red. The locations $l_1,l_2,l_3$ are represented by black points and the elements of the set $\left(l_1 + \mathcal{U}\right) \cup \left(l_2 + \mathcal{U}\right) \cup \left(l_3 + \mathcal{U} \right)$ are colored in gray. We can observe, as expected, that all the local solutions found are contained in the set $C$ given in \eqref{eq:t2 contained in}.

\begin{figure}[th]
\includegraphics[scale=0.8]{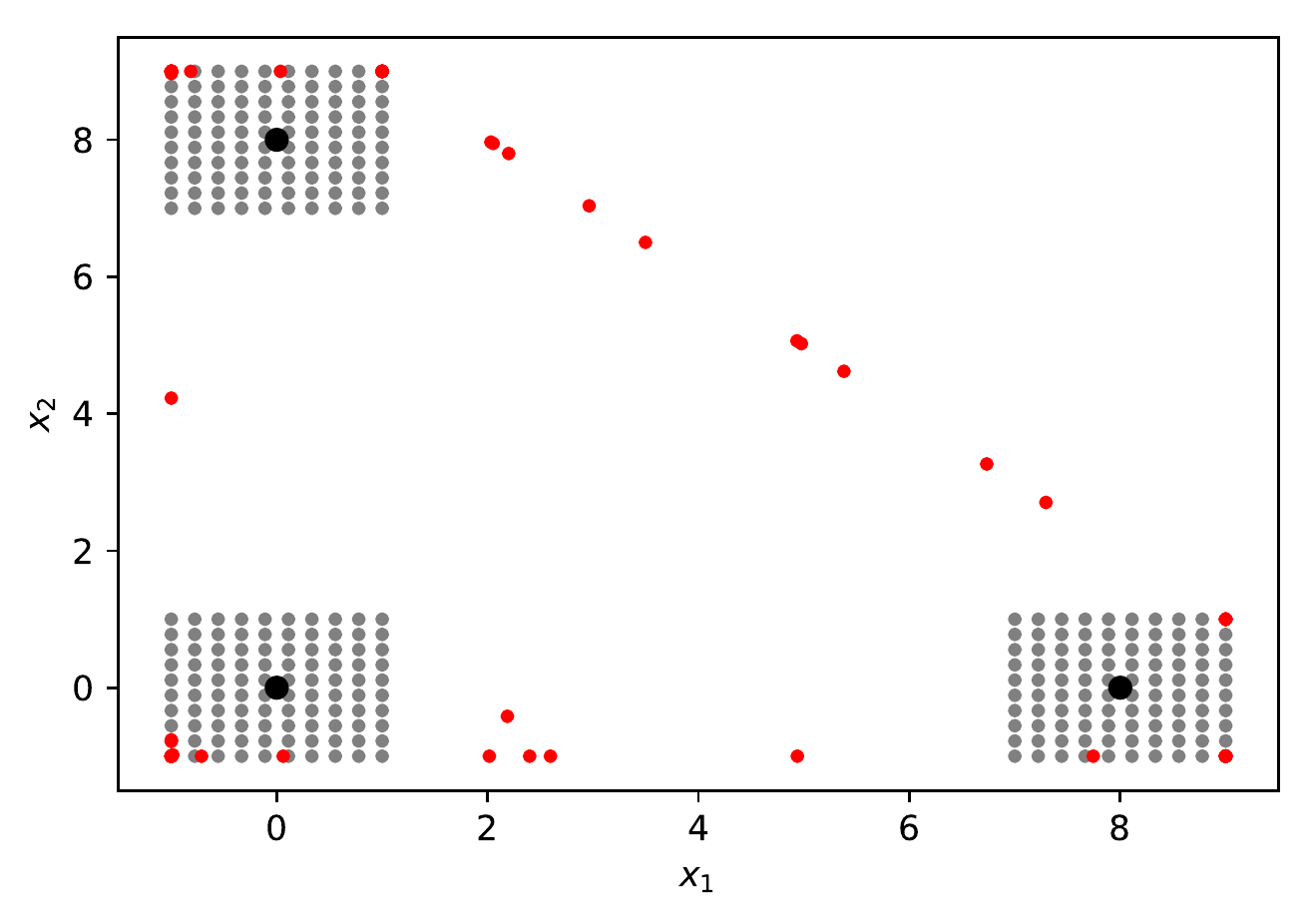}
\centering
\caption{Solutions found (in red) in the argument space for Test Instance \ref{test2}} 
\label{fig2}
\end{figure}

Our last test example comes from \cite{jahn2018tree}.

\begin{Test Instance}\label{test4}  For $i \in [100], $ we consider the functional $f^i: \R^2 \rightarrow \R^2$ defined as
$$f^i(x):= \begin{pmatrix}
e^{\frac{x_1}{2}} \cos(x_2)+ x_1 \cos(x_2) \cos^3\left(\frac{2\pi(i-1)}{100}\right) - x_2\sin(x_2)\sin^3\left(\frac{2\pi(i-1)}{100}\right)\\
e^{\frac{x_2}{20}}\sin(x_1) + x_1 \sin(x_2)\cos^3\left(\frac{2\pi(i-1)}{100}\right) + x_2\cos(x_2)\sin^3\left(\frac{2\pi(i-1)}{100}\right)
\end{pmatrix}.$$ Hence, $F: \R^2 \rightrightarrows \R^2$ is given by 

$$F(x):= \{f^1(x), \ldots , f^{100}(x)\}.$$

\end{Test Instance}

The images of the set-valued mapping  in this example are discretized, shifted, rotated, and deformated rhombuses, see Figure \ref{fig4}. We generated randomly 100 initial points in the square $[-10\pi, 10 \pi] \times [-10\pi, 10 \pi]$ and ran our algorithm. A summary of the results is collected in Table \ref{tab4}. In this case, only for $88$ initial points a solution was found. In the rest of the occasions, the algorithm stopped because the maximum number of iterations was reached. Further examination in these unsolved cases revealed that, except for two of the initial points, the final error was of the order of $10^{-1}$ (even $10^{-3}$ and $10^{-4}$ in half of the cases). Thus, perhaps only a few more iterations were needed in order to declare strong stationarity. 
\begin{figure}[h]
\includegraphics[scale=0.8]{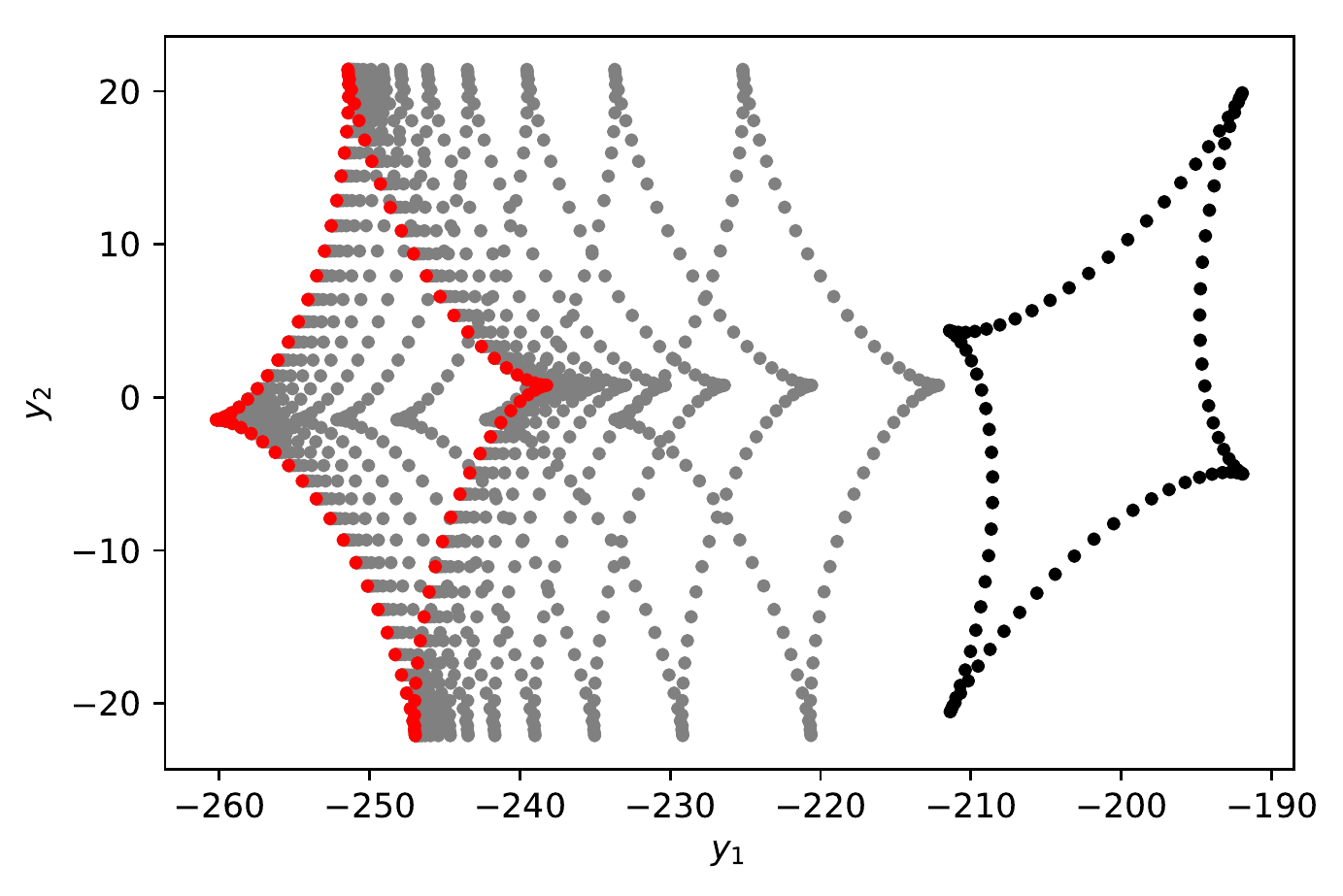}
\centering
\caption{Sequence generated in the image space   by Algorithm \ref{alg:setopt desc} for a selected starting point in Test Instance \ref{test4}.}
\label{fig4}
\end{figure}

%
%
%
%
%
%
%
%

\begin{table}[h]
\centering
\begin{tabular}{ |c|c|c|  }
 \hline
 \multicolumn{3}{|c|}{\textbf{Test Instance \ref{test4}}} \\
 \hline
 \textbf{Solved} & \textbf{Iterations}  & \textbf{Mean CPU Time} \\
 \hline
   88 & (0, 11.9091 , 110)    &  0.8492  \\
 \hline
\end{tabular}

\caption{Performance of Algorithm \ref{alg:setopt desc} in Test Instance \ref{test4}.} 
\label{tab4} 
\end{table}

Figure \ref{fig4} illustrates the sequence $\left\{F(x_k)\right\}_{k \in \{0\} \cup[18]}$ generated  by Algorithm \ref{alg:setopt desc} for a selected starting point. Strong  stationarity was declared after  18 iterations in this experiment. The sets $F(x_0)$ and $F(x_{18})$ are represented by black and red points respectively, and the elements of the sets  $F(x_k)$  with $k \in [17]$  are in gray color. Similarly to the other test instances, we can observe that at every iteration the images decrease with respect to the preorder $\preceq^\ell.$

\section{Conclusions}\label{sec: conclussions}

In this paper, we considered set optimization problems with respect to the lower less set relation, were the set-valued objective mapping can be decomposed into a finite number of continuously differentiable selections. The main contributions are the tailored optimality conditions derived using the first order information of the selections in the decomposition, together with an algorithm  for the solution of the problems with this structure. An attractive feature of our method is that we are able to guarantee convergence towards points satisfying the previously mentioned optimality  conditions. To the best of our knowledge, this would be the first procedure having such property in the context of set optimization. Finally, because of the applications of problems with this structure in the context of optimization under uncertainty, ideas for further research include the development of cutting plane strategies for general set optimization problems, as well as the extension of our results to other set relations.

\nocite{ernestdiss}
\bibliographystyle{siam}
\bibliography{references}
\end{document}